\documentclass[10pt,centertags]{amsart}
\usepackage{amsmath}
\usepackage{amsthm}
\usepackage{amssymb}
\usepackage{mathrsfs} 
\usepackage[pdftex,a4paper]{hyperref}
\usepackage{exscale}
\usepackage{verbatim}
\allowdisplaybreaks
\theoremstyle{definition}
    \newtheorem{definition}{Definition} 
    \newtheorem*{definition*}{Definition}
    \newtheorem{example}[definition]{Example}
    \numberwithin{definition}{section}
\theoremstyle{plain}
    \newtheorem{lemma}[definition]{Lemma}
     
    \newtheorem{proposition}[definition]{Proposition}
    \newtheorem{theorem}[definition]{Theorem}
    \newtheorem*{theorem*}{Theorem}
    \newtheorem{corollary}[definition]{Corollary}

    \newtheorem*{claim*}{Claim}
\theoremstyle{remark}

    \newtheorem{remark}[definition]{Remark}

\def\listofnotations{
\addnotation b_\rho: {cross ratio associated with the maximal representation $\rho$}{CR}
\addnotation B_\Dca: {generalized cross ratio on $\Dca$}{BDca}
\addnotation B_\Dca^{(\alpha)}: {weighted Bergman cross ratio on $\mathcal D$}{BDcaalpha}
\addnotation B_{\check S}: {generalized cross ratio of $\check S$}{BS}
\addnotation \Dca: {bounded symmetric domain}{Dca}
\addnotation \Dca_V: {bounded symmetric domain of a  EuclideanJordan algebra $V$}{DcaV}
\addnotation G: {identity component of the group of biholomorphisms of $\Dca$}{G}
\addnotation \Gamma: {fundamental group of $\Sigma$}{Gamma}
\addnotation K: {stabilizer of $0$ in $G$}{K}
\addnotation k_\Dca: {Bergman kernel of $\Dca$}{kDca}
\addnotation k_V: {normalized kernel function}{kV}
\addnotation k_\chi: {kernel function associated with a character $\chi$}{kchi}
\addnotation \Omega_V: {interior of the cone of squares in a Jordan algebra $V$}{Omega}
\addnotation Q_{\pm}: {standard pair of Shilov parabolics of $G$}{Q}
\addnotation \phi: {limit curve of maximal representation}{LC}
\addnotation \rho: {maximal representation}{MR}
\addnotation {\check S}: {Shilov boundary}{shilov}
\addnotation \Sigma: {closed oriented surface}{Sigma}
\addnotation \tau_\Dca: {translation length}{tau}
\addnotation \tau^\infty_\Dca: {period}{tauinfty}
\addnotation T_\Omega: {tube $V+i\Omega$ over $\Omega$}{TOmega}
\addnotation V: {Euclidean Jordan algebra}{V}
}
\def\addnotation #1: #2#3{\parbox{.3in}{$#1$} \parbox{5in}{#2 \dotfill  \pageref{#3}}\\}
\def\newnot#1{\label{#1}}

\DeclareMathOperator{\tr}{tr}

\DeclareMathOperator{\rk}{rk}

\renewcommand{\bf}{\textbf}

\renewcommand{\phi}{\varphi}
\renewcommand{\rho}{\varrho}

\newcommand{\C}{\mathbb{C}}
\newcommand{\Z}{\mathbb{Z}}

\newcommand{\R}{\mathbb{R}}
\newcommand{\D}{\mathbb{D}}

\newcommand{\Dca}{\mathcal{D}}

\newcommand{\id}{\rm{id}}

\renewcommand{\L}{\mathfrak}

\begin{document}
\bibliographystyle{plain}

\title[Cross ratios]{Cross Ratios, Translation Lengths and Maximal Representations}
\author{Tobias Hartnick}
\address{Technion - Israel Institute of Technology, Mathematics Department\\
         Haifa 32000, Israel}
\email{hartnick@tx.technion.ac.il}
\author{Tobias Strubel}
\address{ETH Z\"urich, Departement Mathematik\\
         R\"amistrasse 101, CH-8092 Z\"urich}
\email{tobias.strubel@math.ethz.ch}
\date{\today}
\begin{abstract} We define a family of four-point invariants for Shilov boundaries of bounded symmetric domains of tube type, which generalizes the classical four-point cross ratio on the unit circle. This generalization, which is based on a similar construction of Clerc and \O rsted, is functorial and well-behaved under
products; these two properties determine our extension uniquely. Our generalized cross ratios can be used to estimate translation lengths of a large class of isometries of the underlying bounded symmetric domain. Our main application concerns maximal representations of surface groups with Hermitian target. For any such representation we can construct a strict cross ratio on the circle in the sense of Labourie via pullback of our generalized cross ratio along a suitable limit curve. In this context our translation length estimates then imply that maximal representations with Hermitian target are well-displacing; this implies in particular that the action of the mapping class group on the moduli space of maximal representations into a Hermitian Lie group is proper.
\end{abstract}
\maketitle

\setcounter{tocdepth}{1}
\tableofcontents

\section{Introduction} This article is concerned with three interrelated problems:
\begin{itemize}
\item[(i)] the development of a functorial theory of generalized cross ratios on Shilov boundaries of bounded symmetric domains of tube type (following work of Clerc and \O rsted \cite{CO2});
\item[(ii)] estimates for the translation length of isometries of bounded symmetric domains of tube type, which have two transversal fixed points in the Shilov boundary, in terms of these cross ratios;
\item[(iii)] applications to maximal representations of surface groups into Hermitian Lie groups (as suggested by earlier work of Labourie \cite{Lab05} and Wienhard \cite{AnnaMCG}).
\end{itemize}
Concerning (i) we recall that the classical four point cross ratio on $\mathbb{CP}^1$ is defined by the formula
\[[a:b:c:d] := \frac{(a-d)(c-b)}{(c-d)(a-b)};\]
its restriction to the circle classifies orbits of ordered quadruples under the actions of ${\rm PSL}_2(\R)$. For boundaries of more general symmetric spaces the space of invariant functions on $4$-tuples will no longer be one-dimensional, hence it is not obvious how to extend the definition of the cross ratio to more general semisimple Lie groups. In fact, it is not even clear what would be the correct notion of boundary to be used in a general theory of cross ratios. Various inequivalent definitions of generalized cross ratios (in different degrees of generality) exist in the literature, see e.g. \cite{CRSiegel, CRBraun, CRKim, CRBiallas} and \cite[Subsec. 4.2.6]{Lab05}. In this article we will consider the situation, where $\mathcal D$ is a bounded symmetric domain of tube type and $G$ is the identity components of its isometry group with respect to the Bergman metric. (See Section \ref{Prelims} for background and definitions.) In this case, a natural choice of boundary for $\mathcal D$ is the Shilov boundary $\check S$, and we will study invariant function on quadruples in $\check S$. Our basic idea is that a good generalization of the classical cross ratio should be functorial (in a sense to be made precise below) and well-behaved under products. If we demand these two properties then there is actually only one choice:
\begin{theorem}\label{Main}
For every bounded symmetric domain $\mathcal D$ of tube type with
Shilov boundary $\check S$ there exists a subset $\check S^{(4+)}$
of $\check S^4$ (defined in Definition \ref{DefExtremal} below)
and a function $B_{\check S}:\check S^{(4+)} \to \R^\times$ called
the \emph{generalized cross ratio} of $\check S$, such that the
family of functions $\{B_{\check S}\}$ is characterized uniquely
by the following properties:
\begin{itemize}
\item[(i)] $B_{\check S}$ is invariant under the group of biholomorphic automorphisms of $\mathcal D$.
\item[(ii)] If $f: \mathcal D_1 \to \mathcal D_2$ is a balanced tight morphism (see Definition \ref{DefBalanced} below),
then the corresponding generalized cross ratios $B_{\check S_1}$, $B_{\check S_2}$ satisfy
\[B_{\check S_2}(\bar f(v_1), \dots, \bar f(v_4)) = B_{\check S_1}(v_1, \dots, v_4),\]
where $(v_1, \dots, v_4) \in \check S_1^{(4+)}$ and $\bar f$ is the boundary extension of $f$.
\item[(iii)] If $\mathcal D = \mathcal D_1 \times \mathcal D_2$ is a direct product of bounded symmetric domains of ranks $r_1, r_2$ with projections $p_j: \Dca \to \Dca_j$ and corresponding
boundary extensions $\bar p_j: \check S \to \check S_j$ then
\[B_{\check S}(v_1, \dots, v_4)^{r_1+r_2} = B_{\check S_1}(\bar p_1(v_1), \dots, \bar p_1(v_4))^{r_1}B_{\check S_2}(\bar p_2(v_1), \dots, \bar p_2(v_4))^{r_2}.\]
\item[(iv)] $B_{S^1}$ is the restriction of the classical four point cross ratio.
\end{itemize}
\end{theorem}

(Theorem \ref{Main} will be proved in Section \ref{MainThmProof} below.)\\

The proof of the theorem is constructive. Cross ratios for \emph{irreducible} bounded symmetric domains of tube type have been constructed by Clerc and \O rsted in \cite{CO2}, and it is easy to modify their construction in such a way that it becomes functorial. The main difficulty is then to show that the extension of these generalized cross ratios to arbitrary bounded symmeric domains by means of (iii) is still functorial. In fact, as will be explained in more details in Section \ref{SubsecBalanced} below, this can only be achieved by restricting the class of admissible morphisms to exclude obvious pathologies.\\

One of the reasons for the importance the classical cross ratio in hyperbolic geometry is the fact that is can be used to define the hyperbolic metric. As a consequence, it can also be used to measure translation lengths of hyperbolic isometries. Indeed, recall that given an isometry $g$ of a metric space $X$ the \emph{translation length} $\tau_X(g)$ is defined by the formula
\[\tau_X(g) := \inf_{x \in X} d(x, gx).\]
For an isometry $\gamma$ of the Poincar\'e disc $\mathbb D$ this translation length is non-zero if and only if $\gamma$ is hyperbolic, i.e. admits a unique repellent fixed point $\gamma^-$ and a unique attractive fixed point $\gamma^+$ in $S^1$. In this case we can compute the translation length of $\gamma$ by the formula
\begin{eqnarray}\label{DiscEquality}
\tau_{\mathbb D}(\gamma) = \tau^\infty_{\mathbb D}(\gamma) := \log [\gamma^-:\xi:\gamma^+:\gamma.\xi],\end{eqnarray}
where $\xi \in S^1Ê\setminus \{\gamma^{\pm}\}$ is an arbitrary auxiliary point. The right hand side of this equation is referred to as the \emph{period} of $\gamma$. Using our generalizd cross ratios we can define a period 
\begin{eqnarray}\label{PeriodsIntro}
\tau^\infty_{\mathcal D}(g, g^+, g^-) := \log B_{\check S}(g^-,\xi,g^+,g.\xi), \quad (\xi \in \check S),
\end{eqnarray}
for every triple $(g, g^-, g^+)$, where $g$ is an isometry of a bounded symmetric domain $\mathcal D$ and $g^{\pm}$ is a pair of transverse fixed points of $g$ in $\check S$. Reordering $g^{\pm}$ if necessary we may assume $\tau^\infty_{\mathcal D}(g, g^+, g^-)  \geq 0$. Without any further assumptions we then find a constant $C_{\Dca}$ depending only on $\mathcal D$ such that (see Corollary \ref{CRTransl} below)
\begin{eqnarray}\label{MainEstimate}
\tau_{\Dca}(g) \geq C_{\Dca} \cdot \tau^\infty_{\Dca}(g).
\end{eqnarray}
Remarkably, no hyperbolicity assumptions on $g$ are required for this inequality to hold. On the other hand, to obtain a similar upper bound for $\tau_{\Dca}(g)$ in terms of $\tau^\infty_{\Dca}(g)$ certain dynamical assumptions on $g$ are necessary. See Corollary \ref{CRTransl} for details.\\

Our main application of Inequality \eqref{MainEstimate} concerns representations of the form \[\rho: \Gamma \to G,\] where $G$ is the automorphism group of a bounded symmetric domain $\mathcal D$ of tube type, and $\Gamma$ \newnot{Gamma} is the fundamental group of a closed surface $\Sigma$\newnot{Sigma}. A particular interesting class of such representations is the class of maximal representations (see \cite{Surface, Anosov, LimitCurves} and the references therein), which can be characterized by the property that there exists a unique equivariant continuous limit curve $\phi: S^1 \to \check S$ subject to a certain monotonicity condition. For such a representation we may define a $\Gamma$-invariant function on quadruples on the circle by the formula
\[b_{\rho}(a,b,c,d) := B_{\check S}(\phi(a), \phi(b), \phi(c), \phi(d)).\]
This function turns out to be a strict cross ratio in the sense of
Labourie \cite{Lab05}, which we refer to as the \emph{strict cross ratio of $\rho$}.\\

By choosing a finite generating set $S$  we can think of the group $\Gamma$ as a metric space with word metric $d_S$. With respect to this metric the translation length of $\gamma \in \Gamma$ on $\Gamma$ is given by the formula
\begin{eqnarray}\label{LengthFunction}
l_S(\gamma):=\inf_{\eta\in \gamma}\|\eta \gamma \eta^{-1}\|_S.
\end{eqnarray}
If we combine the estimate for $b_{\rho}$ arising from \eqref{MainEstimate} with Labourie's equivalence theorem for strict cross ratios (see \cite{Lab05} and Proposition \ref{PropLab}) then we obtain the following relation between $l_S$ and translation length in $\mathcal D$: 
\begin{theorem}\label{ThmResults}
Let $\Gamma$ be the fundamental group of a closed oriented surface $\Sigma$, $\mathcal D$ a bounded symmetric domain and $S$ a finite generating set $S$ for $\Gamma$. Then for every
 maximal representation $\rho: \Gamma \to G(\mathcal D)^0$ there exist $A,B>0$ such that for all $\gamma \in \Gamma$, 
\[
    \tau_{\mathcal D}(\rho(\gamma))\geq A \cdot l_S(\gamma)-B.
\]
\end{theorem}
(Theorem \ref{ThmResults} will be proved in Theorem \ref{WD} below.)\\

In the language of \cite{DGLM} Theorem \ref{ThmResults} says that maximal representations are \emph{well-displacing}, where the constants $A$ and $B$ implicit in this statement depend on the maximal representation in question. This well-displacing property has a number of well-known consequences., which we list briefly. Firstly, given any finite generating set $S$ of $\Gamma$ we can define an associated word metric $d_S$ on $\Gamma$. Then, using results from \cite{DGLM} we obtain:
\begin{corollary}\label{CorQI}
For every $x \in \mathcal D$ and every finite generating set $S$ of $\Gamma$ the map
\[(\Gamma, d_S) \to (\mathcal D, d_{\mathcal D}), \quad \gamma \mapsto \varrho(\gamma).x\]
is a quasi-isometric embedding.
\end{corollary}
Theorem \ref{ThmResults}, Corollary \ref{CorQI} and the Milnor-\v{S}varc lemma (Lemma \ref{MilSva}) imply:
\begin{corollary}\label{CorQI2}
  There exists constants $C,D>0$ such that for all $\gamma \in \Gamma$
  \[
    C^{-1}\tau_\D(\gamma)-D\leq \tau_\Dca(\varrho(\gamma))\leq C\tau_\D(\gamma)+D
  \]  
\end{corollary}

Another consequence of Theorem \ref{ThmResults} concerns the mapping class group of $\Sigma$. Fix a bounded symmetric domain $\mathcal D$ of tube type and denote by $G$ the corresponding automorphism group. The set ${\rm Rep}_{\max}(\Gamma, G)$  of maximal representations of $\Gamma$ into $G$ can be considered of as a subset of $G^S$ for any finite generating set $S$ of $\Gamma$; this induces a locally compact topology on ${\rm Rep}_{\max}(\Gamma, G)$. We denote by $\mathcal M_{\max}(\Gamma, G)$ the quotient of ${\rm Rep}_{\max}(\Gamma, G)$ by the conjugation action of $G$, i.e. the moduli space of conjugacy classes of maximal representations of $\Gamma$ into $G$. Combining Corollary \ref{CorQI2} with results from \cite{AnnaMCG} we obtain:
\begin{corollary}\label{CorProperness}
In the above situation the action of the mapping class group of $\Sigma$ on $\mathcal
M_{\max}(\Gamma, G)$ is proper.
\end{corollary}
For classical simple groups Corollary \ref{CorProperness} was proved by Wienhard
\cite{AnnaMCG} (see also \cite{Lab05} for the symplectic case). Since $\mathcal M_{\max}(\Gamma, {\rm PSL}_2(\R))$ is canonically identified with the Teichm\"uller space of $\Sigma$, we can think of the spaces $\mathcal
M_{\max}(\Gamma, G)$ as \emph{higher Teichm\"uller spaces}. The quotients ${\rm Mod}_g \backslash \mathcal
M_{\max}(\Gamma, G)$ should then be considered as higher ana\-loga of the moduli space of hyperbolic structures on $\Sigma$.\\

As a final application we consider the energy functional of a maximal representation $\rho$ as introduced in \cite{Lab05}:  we denote by $E_\rho:=(\tilde \Sigma \times \Dca)/\Gamma$ the associated $\mathcal D$-bundle over $\Sigma$ and by $\Gamma(E_\rho)$ the space of smooth sections of $E_\rho$. In this notation the energy of a complex structure $J$ on $\Sigma$ with respect to $\rho$ is given by (see \cite[Sec. 5.1]{Lab05})
\begin{eqnarray*}e_{\rho}(J) := \inf \{\int_\Sigma \langle df \wedge df \circ J\rangle\, |\, f \in \Gamma(E_\rho)\}.\end{eqnarray*}
Then $e_\rho$ descends to a functional $e_\rho$ on Teichm\"uller space $\mathcal T(\Sigma)$ called the \emph{energy functional} of $\rho$. In this context, our results imply:
\begin{corollary}\label{CorProperness2}
For any maximal representation $\rho: \Gamma \to G(\mathcal D)$  the associated energy functional $e_\rho: \mathcal T(\Sigma)\to \R$ is proper.
\end{corollary}
(Corollaries \ref{CorQI}-\ref{CorProperness2} will be derived in Subsection \ref{SubsecCorollaries} below.)\\

Let us briefly summarize the structure of this article; for a more
detailed overview over its content see also the introductions to
the individual sections:\\ 

In Section \ref{Prelims} we recall the Jordan algebraic realization of bounded symmetric domains of tube type. We use this opportunity to fix the notation to be used throughout this article. We then define the class of morphisms with respect to which we want to obtain functoriality and characterize them both in Jordan and in Lie theoretic terms. Furthermore we describe the structure of the relevant automorphism groups and collect some results concerning the orbits of transverse triples and quadruples in the Shilov boundary. Finally, we show that the Cayley transform induces a linear representation of the Levi factor of a special maximal parabolic subgroup, which preserves the cone of squares in the associated Euclidean Jordan algebra.\\

The actual construction of our generalized cross ratios is given in Section \ref{SecDefCR}. We first define these functions on the interior of a bounded symmetric domain and provide an algebraic description in terms of a suitable Jordan algebra realization. We then use this description to extend our cross ratio functions to the boundary. The following two sections are then devoted to a study of their properties. Section \ref{SecFun} establishes the desired functoriality; the other main properties are collected in Section \ref{SecPropCR}.\\

Section \ref{SecTL} is devoted to the relation between generalized cross ratios and translation lengths. We first provide bounds for translation lengths of elements of the general linear group acting on the associated symmetric space. Using the linear representation of the Levi factor constructed in Section \ref{Prelims} we thereby obtain bounds for the translation length of special isometries of general bounded symmetric domains of tube type. We then show that these bounds can be expressed in terms of the period of the isometry in question.\\

In Section \ref{SecMaxRep} we introduce the notion of a strict cross ratio and associate a strict cross ratio with every maximal representation. Using Labourie's equivalence theorem for strict cross ratios and the estimates from Section \ref{SecTL} we then establish the well-displacing property of maximal representations. Finally, we indicate how to deduce Corollaries \ref{CorQI}, \ref{CorQI2}, \ref{CorProperness} and \ref{CorProperness2}.\\

For the convenience of the reader we have assembled various facts that are used within the body of the text and which are not readily accessible from the literature in three appendices. Appendix \ref{AppJordan} collects some Jordan theoretic facts used in our proof of the functoriality theorem. Appendix \ref{AppLabourie} contains a formulation of Labourie's equivalence theorem for strict cross ratios, which is particularly well-adapted to the purposes of the present article. This result is implicitly contained in \cite{Lab05}, but since this may not be completely obvious, we decided to include a self-contained proof. Finally, Appendix \ref{AppLimitCurve} establishes a certain uniqueness property of limit curves of maximal representations. This result is a more or less direct consequence of work of Burger, Iozzi and Wienhard in \cite{LimitCurves}. In the preparation of this appendix we profitted from a manuscript on Anosov representations by Anna Wienhard and Olivier Guichard.

\textbf{Acknowledgements} We would like to thank Marc Burger and Alessandra Iozzi for their interest
 in our work and for many useful conversations. We also thank Anna Wienhard and Olivier Guichard for commenting on an earlier version of this article and pointing out various errors;  in particular, Proposition \ref{tightness} arose from discussions with them. We also thank Olivier Guichard for explaining to us the idea of the proof of Lemma \ref{contraction} and for sending us the aforementioned manuscript on Anosov representations. Finally, we would like to thank Kloster Mariastein and Hausdorff Institute Bonn for their
 hospitality during the preparation of this article. The authors were supported by SNF grants PP002-102765 and 200021-127016.

\section{Preliminaries on bounded symmetric domains}\label{Prelims}

In this section we collect some background material on bounded symmetric domains and Euclidean Jordan algebras. Our basic reference is \cite{FK}. We also fix our notation used throughout the text. The first subsection is concerned with the notion of a boundary morphism of a bounded symmetric domain; after recalling the necessary definitions, various characterizations of this notion are presented. We then turn to a description of the corresponding automorphism groups. In the final two subsections we describe orbits of transverse points in the Shilov boundary and a certain linear representation of the Levi factor of a distinguished maximal parabolic.

\subsection{Boundary morphisms of bounded symmetric domains}
Let $W$ be finite-dimensional complex vector space. A connected open subset $\Dca\subset W$ \newnot{Dca} is called a \emph{domain}. A bounded domain $\Dca$ is called \emph{symmetric} if for every $z\in \Dca$ there exists a biholomorphic involutive automorphism $s_z$ of $\Dca$ such that $z$ is an isolated fixed point of $s_z$. A convex,  open $\R^+$-invariant subset $\Omega$ of a real vector space $V$ is called an \emph{open cone}, and in this case the subset $T = V + i \Omega$ of $V \otimes \C$ is called the \emph{tube} over $\Omega$. A bounded symmetric domain is called  \emph{of tube type} if it is biholomorphic to a tube.\\

Recall that for any domain $\Dca$, the Bergman space $\mathcal H^2(\Dca)$ is the space of holomorphic square integrable functions on $W$. If $\Dca$ is bounded then this space is infinite-dimensional and thus the Bergman kernel $k_{\mathcal D}: \mathcal D^2 \to \C^\times$\newnot{kDca} can be defined as its reproducing kernel (see e.g. \cite[Chap. IX.2]{FK}), i.e. by the formula
\[
  f(z)=\int_\Dca f(w)k_\Dca (w,z) dw \quad (f \in \mathcal H^2(\Dca), z \in \mathcal D).
\]
The tensor
\[
  g_{jk}(z):=\frac{\partial^2}{\partial z_j \partial \bar z_k}\log k_\Dca(z,z)
\]
then defines a Hermitian metric on $\Dca$, called the \emph{Bergman metric}, which is invariant under biholomorphisms (see \cite[Prop IX.2.6]{FK}).\\

Given two bounded symmetric domains $\Dca$ and $\Dca'$ with involutions $s_z$ and $s'_{z'}$ respectively, a holomorphic map $f:\Dca \mapsto \Dca'$ is a \emph{morphism} if for any $z\in \Dca$ we have 
\[
  f\circ s_z=s'_{f(z)}\circ f.
\]
Equivalently, $f$ is an \emph{affine} holomorphic map with respect to the Bergman metric on $\mathcal \Dca$. Given a bounded symmetric domain $\mathcal D$, we denote by $G(\mathcal D)$ the group of all automorphisms of $\mathcal D$. Its identity component $G(\mathcal D)^0$ is a finite-dimensional connected adjoint semisimple Lie group acting transitively on $\mathcal D$, and the stabilizer of each point is a maximal compact subgroup. It turns out that all morphisms of bounded symmetric domains are equivariant in the following sense:
\begin{lemma}\label{GroupLift} 
Let $\mathcal D_1, \mathcal D_2$ be bounded symmetric domains and $\beta: \mathcal D_1\to  \mathcal D_2$ a morphism of bounded symmetric domains. Then there exists a finite coverings $\widehat{G}(\mathcal D_1)$ of ${G}(\mathcal D_1)^0$ and a group homomorphism $\widehat{\alpha}:\widehat{G}(\mathcal D_1) \to G(\mathcal D_2)^0$, such that $\beta$ is equivariant with respect to $\widehat{\alpha}$.
\end{lemma}
\begin{proof} Denote by $\widetilde{G}(\mathcal D_1)$ the universal covering of ${G}(\mathcal D_1)^0$. By \cite[Thm. V.1.9]{Bertram} there exists a group homomorphism $\widetilde{\alpha}:\widetilde{G}(\mathcal D_1) \to G(\mathcal D_2)$ with respect to which $\beta$ is equivariant, and it remains to show that $\widetilde{\alpha}$ factors through a quotient of $\widetilde{G}(\mathcal D_1)$ with finite center. For this denote by $\alpha: \L g_1 \to \L g_2$ the induced morphism of Lie algebras; then the complexification $\alpha^\C$ of $\alpha$ lifts to a homomorphism $\alpha^\C: G_1^\C \to G_2^\C$ of the corresponding complex simply-connected groups. Now let $\widehat{G}(\mathcal D_1)$ be the analytic subgroup of $G_1^\C$ with Lie algebra $\L g_1$; then $\widehat{G}(\mathcal D_1)$ is linear (since $G_1^\C$ is), hence has finite center; evidently $\widetilde{\alpha}$ factors through $\widehat{G}(\mathcal D_1)$.
\end{proof}
Given a bounded symmetric domain $\mathcal D \subset W$ we denote by $\check S(\mathcal D) \subset \bar \Dca$ the associated Shilov boundary \newnot{shilov} (see \cite{Clerc} and compare also \eqref{Shilov} below).
\begin{definition}\label{DefShilovExt}
A morphism $\beta:  \mathcal D_1 \to \mathcal D_2$ of bounded symmetric domains with respective Shilov boundaries $\check S_j := \check S(\mathcal D_j)$ is called a \emph{boundary morphism} if it admits a continuous extension $\bar \beta: \check S_1 \to \overline{\mathcal D_2}$ satisfying $\bar \beta(\check S_1) \subset \check S_2$. 
\end{definition}

Note that such an extension, if it exists, is necessarily unique.

Every bounded symmetric domain $\mathcal D$ is isomorphic to the unit ball of a positive Hermitian Jordan triple system $W$ with respect to the spectral norm \cite{Clerc}. If $\mathcal D$ is of tube type, then $W$ can be chosen to be the complexification of a Euclidean Jordan algebra $V$\newnot{V} (see \cite{FK, Loos} and Appendix \ref{AppJordan} for background on Jordan algebras and related notions).  Given a Euclidean Jordan algebra $V$ we denote by $V^\times$ the open subset of invertible elements in $V$ and by $\Omega_V$ (or $\Omega$) the open cone defined by\newnot{Omega}
\begin{eqnarray}\label{Cone}\Omega_V := \{x \in V^\times\,|\, \exists y \in V: x = y^2\}.\end{eqnarray}The corresponding tube $V + i\Omega_V$ in $V^\C := V \otimes \C$ will be denoted $T_{\Omega_V}$ (or simply $T_\Omega$)\newnot{TOmega}. This is biholomorphic to the unit ball $\mathcal D_V$ of $V^\C$ (with respect to the spectral norm). An explicit biholomorphism is provided by the restriction of the Cayley transform 
\[c: D(c) \to D(p), \quad c(w) = i(e+w)(e-w)^{-1},\]
where 
\[D(c) := \{w \in V^\C\,|\,e-w\,{\rm invertible}\}, \quad D(p):= \{z \in V^\C\,|\,z + ie\,{\rm invertible}\},\]
see \cite[Theorem X.1.1]{FK}. The inverse for $c$ is given by
\[p: D(p)\to D(c), \quad p(z) = (z-ie)(z+ie)^{-1}.\]
According to \cite[Thm. X.4.6]{FK} the Shilov bounday $\check S_V$ of $\mathcal D_V$ admits the explicit description
\begin{eqnarray}\label{Shilov}\check S_V = \{z \in V^\C\,|\, z \,{\rm invertible}, z^{-1} = \bar z\}\end{eqnarray}
in terms of $V$. As a consequence, every morphism of Euclidean Jordan algebras induces a boundary morphism of the corresponding unit balls. (Here and in the sequel morphisms between Jordan algebras are assumed unital.) In fact, every boundary morphism of tube type domains arises in this way:
\begin{proposition}\label{tightness}
Let $\mathcal D_1, \mathcal D_2$ be bounded symmetric domains of tube type with respective Shilov boundaries $\check S_1$ and $\check S_2$, and $\beta: \mathcal D_1\to  \mathcal D_2$ be a morphism (i.e. affine holomorphic). Then the following are equivalent:
\begin{itemize}
\item[(i)] $\beta$ is a boundary morphism, i.e. admits a boundary extension satisfying $\bar \beta(\check S_1) \subset \check S_2$.
\item[(ii)] There exist Euclidean Jordan algebras $V_1$, $V_2$, a Jordan algebra homomorphism $\alpha: V_1 \to V_2$ and isomorphisms $\mathcal D_j \cong \mathcal D_{V_j}$ intertwining $\beta$ and $\alpha^\C$.
\item[(iii)] $\beta$ is tight.
\item[(iv)] $\beta$ lifts to a tight homomorphism $\widehat{\beta}:\widehat{G}(\mathcal D_1) \to \widehat{G}(\mathcal D_2)$, where $\widehat{G}(\mathcal D_j)$ is some finite covering of ${G}(\mathcal D_j)^0$.
\end{itemize} 
\end{proposition}
The concept of a tight map between symmetric spaces and their automorphism groups is taken from \cite{Tight}, where the implications 
\[\text{(iii)} \Rightarrow \text{(iv)} \Rightarrow \text{(i)}\]
are proved (see \cite[Cor. 2.16 and Thm. 4.1]{Tight}). As far as the implication $\text{(ii)} \Rightarrow \text{(iii)}$ is concerned, we learned the following argument from O. Guichard: We may assume $\mathcal D_j = \mathcal D_{V_j}$ and $\beta = \alpha^\C$ for some morphism $\alpha: V_1 \to V_2$ of Euclidean Jordan algebras. We then have embeddings of the Poincar\'{e} disc into $\mathcal D_j$ given by
\[\iota_j: \mathbb D \to \mathcal D_j, \quad \lambda \mapsto \lambda \cdot e_j,\]
where $e_j$ is the unit element of $V_j$; these satisfy $\beta \circ \iota_1 = \iota_2$. Now the embeddings $\iota_1$ and $\iota_2$ are tight and positive; however, as proved in \cite[Lemma 8.1]{Tight}, a morphism intertwining positive tight discs is itself tight. This implies (iii). Thus the only missing implication is $\text{(i)} \Rightarrow \text{(ii)}$; for this we provide a Jordan algebraic proof in the appendix (see Proposition \ref{tightnessAppendix}). 

\subsection{The automorphism group} Given a Euclidean Jordan algebra $V$ with unit element $e=e_V$ and associated bounded symmetric domain $\mathcal D_V$ we denote by $G_V$ the identity component of the automorphism group of $\mathcal D_V$ and set $K_V := {\rm Stab}_0(G_V)$ and $Q_{\pm,V} := {\rm Stab}_{\pm e}(G_V)$. We use the small gothic letters $\L g_V, \L k_V, \L q_{+, V}$ to denote  the respective Lie algebras. The group $K_V$ is a maximal compact subgroup of $G_V$ and thus induces a Cartan decomposition $\L g_V= \L k_V \oplus \L p_V$, where $ \L p_V$ is the Killing orthogonal complement of $\L k_V$ in $\L g_V$. In particular, $T_0 \mathcal D_V \cong \L p_V$. The subgroups $Q_{\pm,V}$ are conjugate maximal parabolic subgroups of $G_V$. We refer to the parabolics in their conjugacy class as \emph{Shilov parabolics}. Note that  $Q_{+,V}$ and $Q_{-, V}$ share the same Levi factor $L(Q_{\pm, V}) = Q_{+, V} \cap Q_{-, V}$, which is the pointwise stabilizer of $\{\pm e_V\}$. We will use the notations $G_V, K_V,  Q_{\pm, V}$ throughout this article. Whenever the Jordan algebra $V$ is clear from the context we will simply write $G, K, Q_{\pm}$.  \newnot{G} \newnot{K} \newnot{Q}\\

\subsection{Orbits of transverse points}
Since $Q_+$ is a maximal parabolic in $G$, there is a generalized Bruhat decomposition of $G$ with respect to $Q_+$ (see e.g. \cite[Thm. 7.40]{Knapp}). This allows us to define a notion of transversality on the generalized flag manifold $\check S = G/Q_+$ Namely, two points $z:=gQ_+, w:=hQ_+ \in \check S$ are \emph{transverse}, denoted $z \pitchfork w$, if $Q_+g^{-1}hQ_+$ coincides with
the unique (open) cell of maximal dimension in the Bruhat decomposition of $\check S$ with respect to $Q_+$. For various characterizations of transversality on the Shilov boundary see Proposition \ref{TransMain} in the appendix. We write
\[
  \check S^{(n)} := \{(z_1, \dots, z_n) \in \check S^n\,|\, \forall {i \neq j}: \; z_i \pitchfork z_j\}
\]
for the set of pairwise transverse $n$-tuples in $\check S$. Since the $G$-action preserves transversality, each $\check S^{(n)}$ is a union of $G$-orbits. For $n=2$ we see from
the definition that
$\check S^{(2)}$ is the unique $G$-orbit in $\check S^2$ of maximal dimension. This characterization can be
used to identify $\check S^{(2)}$ in concrete examples. 

\begin{example}
In the case of $G = {\rm Sp}(2n, \R)$ the Shilov boundary is identified with the set $\mathcal L(\R^{2n})$ of Lagrangian
subspaces of $\R^{2n}$. Classically, two Lagrangian subspaces $V,W$ of $\R^{2n}$ are called transverse
if $V \oplus W  = \R^{2n}$. Clearly,
\[\mathcal L(\R^{2n})^{(2)} = \{(V,W) \in \mathcal L(\R^{2n})^2\,|\, V \oplus W  = \R^{2n}\}\]
is an open ${\rm Sp}(2n, \R)$-orbit, hence $\check S^{(2)} = \mathcal L(\R^{2n})^{(2)}$. 
\end{example}

Returning to the general case we recall that $G$-orbits in $\check S^{(3)}$ are classified by the generalized Maslov index $\mu_{\check S}$ of Clerc and \O rsted, see \cite{CO2}. (For a complete classification of orbits in $\check S^3$ see \cite{ClercNeeb}.) Concerning $G$-orbits in $\check S^{(4)}$ we will confine ourselves with the following result. 

Let $\Dca$ be a bounded symmetric domain. Then an affine embedding of $\D^n$ into $\Dca$ is called a \emph{polydisc}; it is called and a \emph{maximal polydisc}, if $n=\rk \Dca$. We will usually not distinguish between the polydisc embedding and its image. Given a Jordan algebra $V$ we refer to a maximal collection $\{c_i\}$ of idempotents of $V$ satisfying $c_ic_j=0 $ for all $ i \neq j$ as a \emph{Jordan frame}.

\begin{proposition}\label{Quadruples}
Let $(z_1, \dots, z_4) \in \check S^{(4)}$, and suppose $\mu_{\check S}(z_i, z_j, z_k)$ is maximal for some $\{i,j,k\} \subset \{1, \dots,
4\}$. Then $z_1, \dots, z_4$ are contained in the boundary of a common maximal polydisc. More precisely, if $\mu_{\check S}(z_1, z_2, z_3)$ is maximal, then there exists $g\in G$ and a Jordan frame $(c_1, \dots, c_r)$ such that 
\[g.(z_1, \dots, z_4) = (\sum (-1) \cdot c_j, \sum(- i) \cdot c_j, \sum 1 \cdot c_j, \sum \lambda_j c_j).\]
\end{proposition}
\begin{proof} Let $r := {\rm rk}(V)$.
We may assume w.l.o.g. that $\mu_{\check S}(z_1, z_2, z_3)$ is maximal, i.e. \[\mu_{\check S}(z_1,z_2,z_3) = r = \mu_{\check S}(-e,
-ie,e).\] Since the Maslov index classifies orbits of transverse triples we then find $g \in G$ with \[g.(z_1,z_2,z_3) = (-e, -ie,e).\] Let $z = g.z_4$. By Proposition \ref{shilovpoints} there
exists a Jordan frame $(c_1,\ldots,c_r)$ and $\lambda_i\in \C$ with $|\lambda_i|=1$ such that
\[
 z=\sum_{i=1}^r\lambda_i c_i.
\]
 We then get the desired equality
\[g.(z_1, \dots, z_4) = (\sum (-1) \cdot c_j, \sum(- i) \cdot c_j, \sum 1 \cdot c_j, \sum \lambda_j c_j)\]
and we deduce that the quadruple is contained in the Shilov boundary of the polydisc
    \[
      \varphi_c: \D^r\rightarrow \Dca, \quad
        (\lambda_1,\ldots,\lambda_r)\mapsto
        \sum_{i=1}^r\lambda_ic_i
    \]
associated with the Jordan frame $c=(c_1, \dots, c_r)$. Consequently, $(z_1,
\dots, z_4)$ is contained in the Shilov boundary of the maximal polydisc $g^{-1} \circ \phi_c$.
\end{proof}
 
 \subsection{The Cayley transform and representations of Levi factors}
To obtain a better understanding of the fine structure of $G$ we observe that the Cayley transform $c: \mathcal D_V \to T_\Omega$ induces an isomorphism
\begin{eqnarray}\label{CayleyInduced}
\hat{c}: G \to G(T_\Omega)^0, \quad g \mapsto c \circ g \circ c^{-1}.
\end{eqnarray}
Denote by $\L g(T_\Omega)$ and $\L g(\Omega)$ the Lie algebras of $G(T_\Omega)$ and \[G(\Omega) := \{g \in GL(V)\,|\, g.\Omega = \Omega\}.\] We will consider $G(\Omega)$ as a subgroup of $G(T_\Omega)$ acting diagonally on $V+i\Omega$ (cf. \cite[p.205]{FK}). Then $\L g(T_\Omega)$ admits a $\Z$-grading with  $\L g(T_\Omega)_0 =  \L g(\Omega)$, $\L g(T_\Omega)_{\pm 1} \cong V$ and $\L g(T_\Omega)_n = \{0\}$ for $|n| > 1$ (see e.g  \cite[Sec. 6]{LawsonLim}). We will denote by $N^{\pm}$ the analytic subgroups of $G(T_\Omega)^0$ corresponding to $\L g(T_\Omega)_{\pm 1}$.  Then 
$G(\Omega)$ normalizes $N^\pm$ and we can thus form the semidirect products $P^+ := N^- G(\Omega)$ and $P^- :=N^+G(\Omega)$. (The reason for these sign conventions will become clear in Proposition \ref{Levi}.) It turns out that $P^{\pm}$ are maximal parabolic subgroups of $G(T_\Omega)^0$ and that $P^-$ stabilizes $0 \in V$. Its unipotent radical is given by $N^+$ and its Levi factor is given by $G(\Omega)$ (see \cite[Sec. 7]{LawsonLim}). Now we have:
\begin{proposition}\label{Levi}
Let $\widehat{c}: G \to G(T_\Omega)^0$ be the isomorphism given by \eqref{CayleyInduced}. Then $\widehat{c}(Q_-) = P^-$ and $\widehat{c}(L(Q_-)) = G(\Omega)$.
\end{proposition}
\begin{proof} Since $P^-$ stabilizes $0 \in V$ the group $\widehat{c}^{-1}(P^-)$ stabilizes $c^{-1}(0) = -e$. 
Thus $\widehat{c}^{-1}(P^-) \subset Q_-$ is a subgroup, but being maximal parabolic itself we find $\widehat{c}^{-1}(P^-) = Q_-$. Passing to the corresponding Levi factors yields the second statement.
\end{proof}
For later reference we record the following consequences:
\begin{corollary}\label{LeviRep}
\begin{itemize}
\item[(i)] The unipotent radical of a Shilov parabolic is abelian.
\item[(ii)] The map $\widehat{c}$ provides a linear representation $\widehat{c}: L(Q_{\pm}) \to GL(V)$ for the Levi factor of the standard Shilov parabolics.
\end{itemize}
\end{corollary}
We will exploit the linear representation of $L(Q_{\pm})$ in Section \ref{SubsecTLLinear} below to estimate translation lengths.

\section{Construction of generalized cross ratios}\label{SecDefCR}

In this section we define the protagonists of this article, namely generalized cross ratios on bounded symmetric domains of tube type. Our definition proceeds in three steps: In the first subsection we define a generalized cross ratio for four-tuples of points inside a bounded symmetric domain. We then provide in the second subsection an algebraic description of these generalized cross ratios. In the final subsection we use this description to prove that our generalized cross ratio extends to certain four-tuples on the Shilov boundary.\\

A remark concerning our normalizations seems in place here: The cross ratios defined here are special cases of more general parameter-dependent cross ratios; imposing functoriality automatically fixes these parameters. To keep this exposition simple we refrained from carrying the parameters along; instead we decided to fix the correct parameters a priori. We hope that Proposition \ref{PropFunctoriality} will convince the reader that a posteriori our normalization is the correct one. To give the reader an idea of the normalizations involved, consider the case of irreducible bounded symmetric domains. If $\phi:\Dca_1 \rightarrow \Dca_2$ is a morphism of such domains and $k_1$, $k_2$ are the associated kernels as defined by Clerc and \O rsted in \cite{CO2}, then \cite[Prop. 6.2]{CO2}
\[
  k_2(\phi(x),\phi(y))^{r_1}=k_1(x,y)^{r_2},
\]
where $r_i$ are the respective ranks. Thus to obtain a functorial kernel function on $\mathcal D$ one should consider the $({\rm rk}\, \mathcal D)$th roots of the kernel functions of Clerc and \O rsted (which up to a constant coincides with the $(2 \cdot \dim_\C \mathcal D)$th root of the inverse of the Bergman kernel). These kind of obvious normalizations lead to the definitions presented below.

\subsection{Definition and basic invariance properties}

Let $\mathcal D$ be a domain in a complex vector space $W$, which is biholomorphic to a bounded domain, with Bergman kernel $k_\Dca$ (cf. p. \pageref{kDca}).  Since $\mathcal D^2$ is simply-connected, the rational powers $k_{\mathcal D}^\alpha$ can be defined for any $\alpha \in \mathbb Q$ in $\Dca$; indeed, given $\alpha = \frac p q$ with integers $p, q \in \Z$ we define
$k_{\mathcal D}^{\alpha}$ to be the unique continuous function on $\mathcal D^2$ satisfying
\begin{eqnarray}
(k_{\mathcal D}^{\alpha})^q = k_{\mathcal D}^p, \quad k_{\mathcal D}^{\alpha}(0,0) = 1.
\end{eqnarray}
We then define:
\begin{definition} Let $\mathcal D$ be a domain in a complex vector space $W$, which is biholomorphic to a bounded domain. Then the \emph{weighted Bergman cross ratio}\newnot{BDcaalpha} in $\Dca$ of \emph{weight} $\alpha \in \mathbb Q$ is the function
\begin{eqnarray}
B_{\mathcal D}^{(\alpha)}: \mathcal D^4 \to \C^\times, \quad (x,y,z,t) \mapsto \frac{k_{\mathcal D}^{\alpha}(t,x)k_{\mathcal D}^{\alpha}(y,z)}{k_{\mathcal D}^{\alpha}(t,z)k_{\mathcal D}^{\alpha}(y,x)}.
\end{eqnarray}
\end{definition}
Our first observation is the following crucial invariance property:
\begin{proposition}\label{BergmanCrossInvarianceInterior} Let $\mathcal C$ be complex domains biholomorphic to a bounded domain and let  $c: \mathcal D \to \mathcal C$ be a biholomorphism. Then for all $(x,y,z,t) \in \mathcal D^{4}$ and for every $\alpha \in \mathbb Q$ we have
\[B_{\mathcal D}^{(\alpha)}(x,y,z,t) = B^{(\alpha)}_{\mathcal C}(c(x), c(y), c(z), c(t)).\]
\end{proposition}
\begin{proof} Since the equality is invariant under taking rational powers, it suffices to prove the proposition for $\alpha = 1$. According to \cite[Prop. IX.2.4]{FK} the Bergman kernels
on $\mathcal D$ and $\mathcal C$ are related by the formula,
\[k_{\mathcal D}(z,w) = k_{\mathcal C}(c(z), c(w))\det{}_\C(J_c(z))\overline{\det{}_\C(J_c(w))},\]
where $J_c$ denotes the complex Jacobian of $c$. Thus,
\begin{eqnarray*}
B_{\mathcal D}^{(1)}(x,y,z,t) &=& B_{\mathcal C}^{(1)}(c(x), c(y), c(z), c(t)),
\end{eqnarray*}
since the Jacobian terms cancel.
\end{proof}
In particular we have:
\begin{corollary}\label{CorInvariance} If $\mathcal D$ is a complex bounded domain, then $B_{\mathcal D}^{(\alpha)}$ is invariant under the group $G(\mathcal D)$ of biholomorphic automorphisms of $\mathcal D$ for every $\alpha \in \mathbb Q$.
\end{corollary}
For a general bounded domain we do not see any preferable normalization for $\alpha$; however, for bounded symmetric domains, there is essentially (i.e. up to global constant) only one normalization, which yields the desired functoriality. This normalization is given as follows:
\begin{definition}\label{DefBD}
Let  $\mathcal D$ be a bounded symmetric domain of complex dimension $n$. Then the \emph{generalized cross ratio} of $\mathcal D$ is the function \newnot{BDca}
\[B_{\mathcal D}: \mathcal D^4 \to \C^\times\]
defined as follows: If $\Dca$ is irreducible, then  $B_{\mathcal D} := B_{\mathcal D}^{(-\frac{1}{2n})}$. If $\mathcal D = \mathcal D_1 \times \dots \times \mathcal D_m$ with irreducible factors $D_1, \dots, D_m$ then we define $B_{\mathcal D}$ by the formula 
\[B_{\mathcal D}^{\rk \Dca} = \prod_{i=1}^m B_{\mathcal D_i}^{\rk \mathcal D_i}, \quad B_{\mathcal D}(0,0,0,0) = 1.\]
\end{definition}
\begin{remark}
The appearance of the dimension factor $n$ in the normalization is essential for the functoriality of the cross ratio, see the proof of Lemma \ref{UniversalitySimpleCase}. On the other hand, the additional factor $2$ in the denominator is purely for reasons of normalization, see Example \ref{CRClassic}. The reasons for weighting the simple factors in the present way are more subtle; see the proof of Proposition \ref{BalancingTheorem}.
\end{remark}
While $B_{\mathcal D}$ is not a weighted Bergman cross ratio in the strict sense, it still inherits the following property:
\begin{corollary}\label{kVProjection}
Let $\mathcal D$ be a bounded symmetric domain. Then $B_{\mathcal D}$ is invariant under $G(\mathcal D)$. Moreover, if $\mathcal D = \mathcal D_1 \times \mathcal D_2$ is the product of two bounded symmetric domains of respective ranks $r_1, r_2$ with projections $p_j: \mathcal D \to \mathcal D_j$ then
\[B_{\mathcal D}(x,y,z,t)^{r_1+r_2} = B_{\mathcal D_1}(p_1(x),p_1(y),p_1(z),p_2(t))^{r_1}B_{\mathcal D_2}(p_2(x),p_2(y),p_2(z),p_2(t))^{r_2}.\]
\end{corollary}
We remark that the definition of the generalized cross ratio makes sense for any bounded symmetric domain, regardless whether it is of tube type or not. However, in the non-tube type case the framework of Euclidean Jordan algebras is not available and so we would have to use more general Jordan triple systems in order to obtain an algebraic description. Since our applications are only concerned with the tube type case, we decided to avoid this.

\subsection{Algebraic description} Let $\mathcal D$ be a bounded symmetric domain of tube type. By Proposition \ref{BergmanCrossInvarianceInterior} the generalized cross ratio of $\mathcal D$ does not depend on the concrete realization of $\mathcal D$, hence we choose to realize $\mathcal D$ as the bounded symmetric domain  $\mathcal D_V$\newnot{DcaV} of a Euclidean Jordan algebra, i.e. we fix an isomorphism $\mathcal D \cong \mathcal D_V$. We now aim to describe the weighted Bergman cross ratio of $\mathcal D$ algebraically in terms of $V$. For this we introduce the following notions:\\

Let $V$ be a Euclidean Jordan  algebra. Given $z \in V^\C$ we
denote by $L(z)$ the left-multiplication by $z$. Then for all $z,w
\in V^\C$ the \emph{box operator} and the \emph{quadratic
representation} are defined by
\[z \square w := L(zw) + [L(z), L(w)],\]
and
\[P(z) := 2L(z)^2 - L(z^2)\]
respectively. Following \cite{CO2} (see also \cite{FK} and \cite{Satake}) we
define the \emph{automorphy kernel}
\[K: V^\C \times V^\C \to {\rm End}(V^\C),\]
by
\[K(z,w) := I - 2z\square \overline{w} + P(z) P(\overline{w}).\]
We also use the quadratic representation to define the \emph{structure group} of $V^\C$ to be
\[{\rm Str}(V^\C) := \{g \in GL(V^\C)\,|\,P(gx) = gP(x)g^{\top}\},\]
where $g^\top$ is the transpose of $g$ with respect to the Euclidean structure on $V^\C$.
If $z,w \in \mathcal D$ then $K(z,w) \in {\rm Str}(V^\C)$ \cite[p.
315]{CO2}. Thus for every character $\chi: {\rm Str}(V^\C) \to \C^\times$ we obtain a kernel function
\begin{eqnarray}
k_{\chi}: \mathcal D^2 \to  \C^\times, \quad (a,b) \mapsto \chi(K(a,b)).
\end{eqnarray}\newnot{kchi}
By \cite[Prop. X.4.5]{FK} there exists a constant $C = C(V)$ such that
\[k_{\mathcal D_V} = C \cdot k_{\det^{-1}};\]
in particular we get the following \emph{algebraic} description of the weighted Bergman cross ratio in the tube type case:
\begin{eqnarray}\label{BergmanPrelim}B_{\mathcal D_V}^{(\alpha)} = \frac{k_{\det^{-\alpha}}(d,a)k_{\det^{-\alpha}}(b,c)}{k_{\det^{-\alpha}}(d,c)k_{\det^{-\alpha}}(b,a)}.\end{eqnarray}
It is then clear how to define a normalized kernel function $k_V: \mathcal D^2 \to  \C^\times$ such that the generalized cross ratio of $\mathcal D_V$ takes the form
\begin{eqnarray}\label{CRIntFormula}
B_{\mathcal D_V}(a,b,c,d) = \frac{k_V(d,a)k_V(b,c)}{k_V(d,c)k_V(b,a)}.
\end{eqnarray}
Indeed, we define:
\begin{definition}\label{DefNormKernel}
The \emph{normalized kernel function} $k_V: \mathcal D^2 \to  \C^\times$ \newnot{kV}   is defined as follows: If $V$ is simple, then we define $k_V$ to be the unique function satisfying 
\[k_V(z,w)^{2 \dim V} = k_{\det^{-1}}(z,w), \quad k_V(0,0) = 1.\]
For general $V$, we decompose $V = V_1 \oplus \dots
\oplus V_n$ into simple ideals and identify elements $z \in V$ with vectors $z = (z_1, \dots, z_n)^\top$ with $z_j \in V_j$ and define by
\begin{eqnarray*}\label{ProductFormula}k_V(z,w)^{\rk V} = \prod_{i=1}^n (k_{V_i}(z_i, w_i))^{{\rk V_i}}, \quad k_V(0,0) = 1.\end{eqnarray*}
With this definition of $k_V$ the equality \eqref{CRIntFormula} is a direct consequence of \eqref{BergmanPrelim}.
\end{definition}
\begin{example}\label{CRClassicInterior}
Let $V = (\R, \cdot)$ so that $V^\C=\R^\C=\C$ and $\mathcal D_V = \mathbb D$ is the Poincar\'{e} disc. Then for $x,w,z\in \C$ we have
\begin{eqnarray*}
 (z\square w)x&=&L(zw)x+[L(z),L(w)]x=(zw)x,\\
 P(z)x&=&(2L(z)^2-L(z^2))x=z^2x,\\
 K(z,w)x&=&x-2z\bar w x+z^2\bar w^2=(1-z\bar w)^2x,\\
\end{eqnarray*}
in particular $k_\R(z,w) = 1-z\bar w$ and thus
\[B_{\mathbb D}(a,b,c,d) = \frac{(1-d\bar a)(1-b\bar c)}{(1-d\bar c)(1-b\bar a)}.\]
\end{example}
\subsection{Transversality and boundary extensions}
So far we have considered cross ratios in the interior of a bounded symmetric domain $\mathcal D$; our definition relied on the fact that the normalized kernel function does not vanish for any pair in $\mathcal D^2$. Now we want to extend our cross ratio continuously to pairs in  the Shilov boundary $\check S$ of $\mathcal D$; it is indeed possible to extend the normalized kernel function to the topological closure of $\mathcal D$, but the resulting function will have zeros. As far as points $z,w$ in the Shilov boundary are concernced we deduce from Proposition \ref{TransMain} that $\det K(z,w) = 0$ if and only if $z$ and $w$ are not transverse. This observation allows us to extend $k_V$ to a continuous nowhere-vanishing function on $\check S^{(2)}$. To extend it even further, we observe:
\begin{lemma}\label{Branching} Let $X$ be a manifold, $f: X \to \C$ be a continuous function, $X' := f^{-1}(\C \setminus\{0\})$. Let $\widetilde{f}: X' \to \C
\setminus \{0\}$ be any continuous function with $\widetilde{f}^n = f|_{X'}$. Then $\widetilde{f}$ extends continuously by $0$ to
all of $X$.
\end{lemma}
\begin{proof} Extend $\widetilde{f}$ to all of $X$ by $0$. We show that this extension is continuous. For this
let $x_k \in X'$ with $x_k \to x$, where $x \in X\setminus X'$. Then $f(x_k) \to f(x) = 0$ by continuity of
$f$, hence $\widetilde{f}(x_k)^n \to 0$. This, however, implies already $\widetilde{f}(x_k) \to 0 = \widetilde{f}(x)$, which
yields continuity of the extended function.
\end{proof}
Thus we deduce:
\begin{corollary}\label{ExtendKernelTransvWeak} Let $V$ be a Euclidean Jordan algebra. Then the normalized kernel $k_V$ extends continuously to the Shilov boundary and for $z,w \in \check S$ we have
\[k_V(z,w) \neq  0 \Leftrightarrow z \pitchfork w.\]
\end{corollary}
We see in particular from \eqref{CRIntFormula}, that  $B_{\mathcal D}$ extends continuously to a function
\[
B_{\check S}: \check S^{(2)} \times \check S^{(2)} = \{(x,y,z,t) \in \check S^4\,|\, x \pitchfork y, z \pitchfork t\} \to \C,
\]
which is nonzero on $\check S^{(4)}  \subset \check S^{(2)} \times \check S^{(2)}$. It turns out, however, that the present domain for $B_{\check S}$ is too large for our purposes: Neither is the extended cross ratio real-valued on $\check S^{(2)} \times \check S^{(2)}$, nor can we show functoriality for these domains. It turns out, a posteriori, that the following domain is ideally suited for our purposes:
 \begin{definition}\label{DefExtremal}
Let $\mathcal D$ be a bounded symmetric domain and $\check S$ the associated Shilov boundary. A quadruple $(x,y,z,t) \in \check S^{(4)}$ is called \emph{extremal} if any triple $(a,b,c) \
\in \check S^3$ of pairwise distinct points with $a,b,c \in \{x,y,z,t\}$ has either maximal or minimal Maslov index. (Such a triple is then called \emph{maximal} or \emph{minimal} accordingly.) We denote the set of extremal quadruples in $\check S^4$ by $\check
S^{(4+)}$.\\
The \emph{generalized cross ratio} of the Shilov boundary $\check S$ is the function 
\begin{eqnarray}\label{CRKernel}\newnot{BS}
B_{\check S}: \check
S^{(4+)} \to \C^\times, \quad (x,y,z,t)\mapsto \frac{k_V(t,x)k_V(y,z)}{k_V(t,z)k_V(y,x)}.
\end{eqnarray}
\end{definition}
The term \emph{generalized} refers to the following example:
\begin{example}\label{CRClassic}
Consider the Shilov boundary $S^1$ of the Poincar\'{e} disc. We see from Example \ref{CRClassicInterior} that \[B_{S^1}(a,b,c,d) = \frac{(1-d\bar a)(1-b\bar c)}{(1-d\bar c)(1-b\bar a)} = \frac{(a-d)(c-b)}{(c-d)(a-b)} = [a:b:c:d].\]
Similarly, if $\mathcal D = \mathbb D^r$ is a rank $r$ polydisc, then a similar computation (or Lemma \ref{FrameFormulaKernel} below) shows that
\[B_{(S^1)^r}(a,b,c,d) = \left( \prod_{i=1}^r \frac{(a_i-d_i)(c_i-b_i)}{(c_i-d_i)(a_i-b_i)}\right)^{1/r}.\]
In particular, the cross ratio is invariant under the diagonal embedding of $\mathbb D$ into $\mathbb D^r$. This is a first instance of functoriality, which in particular explains our normalization in the reducible case.
\end{example}
We record for later use that the properties of $B_{\mathcal D}$ listed in Corollary \ref{kVProjection} extend by continuity to the boundary extension $B_{\check S}$:
\begin{proposition}\label{AutInvariance}
Let $\mathcal D$ be a bounded symmetric domain of tube type and $\check S$ its Shilov boundary. Then $B_{\check S}$ is invariant under $G(\mathcal D)$. Moreover, if $\mathcal D = \mathcal D_1 \times \mathcal D_2$ is the product of two bounded symmetric domains of respective ranks $r_1, r_2$ with corresponding Shilov boundaries $\check S, \check S_1, \check S_2$ and $p_j: \check S \to \check S_j$ denotes the projection, then
\[B_{\check S}(x,y,z,t)^{r_1+r_2} = B_{\check S_1}(p_1(x),p_1(y),p_1(z),p_2(t))^{r_1}B_{\check S_2}(p_2(x),p_2(y),p_2(z),p_2(t))^{r_2}.\]
\end{proposition}
\begin{remark}
We warn the reader that our kernel function $k_V$ is different from the kernel function denoted $k$ in \cite{CO2}. For simple $V$ the two kernels are related by the formula
\begin{eqnarray}\label{COComparison}
k_V^{2 \cdot {\rm rk}(V)} = k,
\end{eqnarray}
as follows from \cite[Prop.III.4.3]{FK}. For general $V$ the relation is more complicated.
\end{remark}

\section{Functoriality of generalized cross ratios}\label{SecFun}
The goal of this section is the following functoriality result, which a posteriori justifies our normalizations:
\begin{proposition}\label{PropFunctoriality} Let $\mathcal D_1, \mathcal D_2$ be bounded symmetric domains of tube type with respective Shilov boundaries $\check S_1, \check S_2$, let $\beta: \mathcal D_1 \to \mathcal D_2$ be a balanced tight morphism and $\bar \beta: \check S_1 \to \check S_2$ its boundary extension. Then for all $(x,y,z,t) \in \check S^{(4+)}$ we have
\[B_{\check S_2}(\bar \beta(x), \dots, \bar \beta(t)) = B_{\check S_1}(x, \dots, t).\]
\end{proposition}
The notion of a balanced tight morphism will be explained in Definition \ref{DefBalanced} below (see also Example \ref{ExamplesBalanced}).\\

The remainder of this section is devoted to the proof of Proposition \ref{PropFunctoriality}; the reader who is willing to take the proposition on faith, can skip this part except for Definition \ref{DefBalanced}. We have divided the proof into three main steps. We first proof that in the simple case the normalized kernel functions themselves are invariant under suitable morphisms. Such a result is not true for reducible domains, but using a reduction to the irreducible case, we can still obtain a partial invariance result (see Proposition \ref{BalancingTheorem}). This will be sufficient to finally deduce the proposition.

\subsection{The simple case: Invariance of the normalized kernels} 
Throughout this subsection we fix simple Euclidean Jordan algebra $V_1, V_2$ and denote by $\mathcal D_j$ and $\check S_j$, $j=1,2$, the corresponding bounded symmetric domains and their Shilov boundaries respectively. In view of \eqref{CRKernel} the functoriality properties of $B_{\check S_j}$ are closely related to the transformation behaviour of the normalized kernel functions under morphisms. In the simple case, this behaviour is easy to describe:
\begin{lemma}[Clerc-\O rsted]\label{UniversalitySimpleCase}
Let $\alpha: V_1 \to V_2$ be a morphism of simple Euclidean Jordan algebras. Denote by  $\mathcal D_1$ and $\mathcal D_2$ respectively the corresponding bounded symmetric domains
and by $\check S_1$ and $\check S_2$ the respective Shilov boundaries. Then for all $z,w \in\mathcal
D_1 \cup \check S_1$ we have
\begin{eqnarray}\label{UniversalityKernel1} k_{V_2}(\alpha^\C(z), \alpha^\C(w)) = k_{V_1}(z,w).\end{eqnarray}
\end{lemma}
\begin{proof} By continuity it suffices to prove \eqref{UniversalityKernel1} for $z, w \in \mathcal D$. Denote by $r_j$ the rank of $V_j$ and let $k_j := k_{V_j}^{2r_j}$. In view of \eqref{COComparison}, these are precisely the kernel functions from \cite{CO2}, whence \cite[Prop. 6.2]{CO2} yields 
\[k_{2}(\alpha^\C(z), \alpha^\C(w)) = k_{1}(z,w)^{\frac{r_2}{r_1}}.\]
This yields immediately
\begin{eqnarray*}
k_{V_2}(\alpha^\C(z), \alpha^\C(w)) 
&=& k_{2}(\alpha^\C(z), \alpha^\C(w))^{\frac{1}{2r_2}} =\left(k_{1}(z,w)^{\frac{r_2}{r_1}}\right)^{\frac{1}{2r_2}}\\
&=&k_{1}(z,w)^{\frac{1}{2r_1}}
= k_{V_1}(z,w).
\end{eqnarray*}
\end{proof}
We have been slightly sloppy here by not specifying the arc of the various roots. Strictly speaking we have only shown that
\[k_{V_2}(\alpha^\C(z), \alpha^\C(w))^{2r_1r_2} = k_{V_1}(z,w)^{2r_1r_2}.\]
However, in view of $k_{V_2}(\alpha^\C(0), \alpha^\C(0)) = k_{V_1}(0,0)$ this is actually enough to deduce
\[k_{V_2}(\alpha^\C(z), \alpha^\C(w)) = k_{V_1}(z,w).\]
We will allow ourselves this kind of sloppyness regarding roots,
whenever it is clear how to make the arguments precise.

\subsection{Balanced morphisms}\label{SubsecBalanced}
Lemma \ref{UniversalitySimpleCase} does not extend to the general case; in fact, we have the following generic counterexample:
\begin{example}\label{BalancedCounterexample}
Consider the Jordan algebra embedding $\alpha: \R^2 \to \R^3$ given by $(\lambda_1, \lambda_2) \mapsto (\lambda_1, \lambda_1, \lambda_2)$. Then
\begin{eqnarray*}k_{\R^2}(\lambda, \mu) &=& (1 - \lambda_1 \overline{\mu_1})^{\frac 1 2}(1 - \lambda_2 \overline{\mu_2})^{\frac 1 2}\\ &\neq& (1 - \lambda_1 \overline{\mu_1})^{\frac 2 3}(1 - \lambda_2
\overline{\mu_2})^{\frac 1 3} =k_{\R^3}(\alpha^\C(\lambda),\alpha^\C(\mu)).\end{eqnarray*}
\end{example}
We want to exclude bad behavior as in the last example. We  denote by $\tr_V$ the Jordan algebra trace of $V$ and remind the reader that \cite[Thm. III.1.2]{FK} for any Jordan frame $(c_1, \dots,
c_r)$ of $V$ we have
\[x = \sum_{j=1}^r \lambda_jc_j \Rightarrow \tr_V(x) = \sum_{j=1}^r \lambda_j.\]
Now we define:
\begin{definition}\label{DefBalanced} A Jordan algebra homomorphism $\alpha: V \to W$ is called \emph{balanced} if for all $v \in V$
\[\frac{1}{\rk V}\tr_V(v) = \frac{1}{\rk W} \tr_W(\alpha(v)).\]
A tight morphism $\beta: \mathcal D_1 \to \mathcal D_2$ is called \emph{balanced} if there exists Jordan algebras $V, W$ and isomorphisms $\mathcal D_1 \cong \mathcal D_V$ and $\mathcal D_2 \cong \mathcal D_W$ intertwining $\beta$ with the complexification of a balanced morphism of Euclidean Jordan algebras.
\end{definition}
The notion is clearly invariant under complexification. Note that nonzero idempotents have positive trace and thus go to nonzero idempotents under balanced morphisms; this shows that every balanced Jordan algebra homomorphism is injective. Moreover, we have the following characterization of balanced Jordan algebra homomorphisms: Let $(c_1, \dots, c_r)$ be a
Jordan frame in $V$ and $\alpha: V \to W$ a Jordan algebra homomorphism. Then $\alpha(c_1), \dots, \alpha(c_r)$ is a family of idempotents with $\alpha(c_i)\alpha(c_j) = 0$ and
$\sum \alpha(c_i) = e$. By Lemma \ref{JordanCompletion} we thus find a Jordan frame $(c_{11}, \dots, c_{1l_1}, \dots, c_{r1}, \dots, c_{rl_r})$ of $W$ such
that
\[\alpha(c_j) = \sum_{k=1}^{l_j} c_{jk}.\]
We have $\tr_W(\alpha(c_j)) = l_j$ and thus $\alpha$ is balanced
if and only if
\[l_1 = \dots = l_r.\]
Conversely, if the latter condition is true for any Jordan frame $(c_1, \dots, c_r)$ of $V$, then $\alpha$ is balanced. Note that we obtain in particular
\[\rk W = l_j \cdot \rk V \quad (j = 1, \dots, r),\]
so that $\rk W$ is divisible by $\rk V$. The
morphism in Example \ref{BalancedCounterexample} clearly violates this condition, and thus is not balanced. In our attempts to prove an invariance
theorem we will restrict attention to balanced Jordan algebra
homomorphisms. Even in this case we cannot quite obtain the same
kind of invariance as in Lemma \ref{UniversalitySimpleCase}. In
order to formulate our weaker result, we introduce the
following terminology: Two elements $v_1, v_2$ are called
\emph{co-diagonalizable} if there exists a Jordan frame $(c_1,
\dots, c_r)$ and elements $\lambda_j \in \mathbb D$, $\mu_j \in
\mathbb D$ such that
\[v_1 = \sum_{j=1}^r \lambda_jc_j \in \mathcal D_{V}, \quad v_2 = \sum_{j=1}^r \mu_jc_j\in \mathcal D_{V}.\]
By \cite[X.2.2]{FK} $x$ and $y$ are co-diagonalizable if and only if $[L(x),L(y)]=0$. Then we have:
\begin{proposition}\label{BalancingTheorem}
Let $\alpha: V \to W$ be an injective homomorphism of Euclidean Jordan algebras. If $\alpha$ is balanced, then for every pair of
co-diagonalizable elements $v_1, v_2 \in \mathcal D$ we have
\begin{eqnarray}\label{BalancingDream} k_W(\alpha^\C(v_1),\alpha^\C(v_2)) = k_{V}(v_1,v_2).\end{eqnarray}
Conversely, if \eqref{BalancingDream} holds for all co-diagonalizable $v_1, v_2 \in \mathcal D_V$, then $\alpha$ is balanced.
\end{proposition}
For the proof we consider first the case, where $V$ is a maximal polydisc in $W$. In this case we have the following version of Proposition \ref{BalancingTheorem}, which is a slight
extension of the results of Clerc and {\O}rsted in \cite{CO2}:
\begin{lemma}\label{FrameFormulaKernel}
If $(c_1, \dots, c_r)$ is a Jordan frame in a Euclidean Jordan algebra $W$ of rank $r$ and $\lambda_j \in {\mathbb D}$, $\mu_j \in {\mathbb D}$, then
\begin{eqnarray}k_{W}(\sum_{j=1}^r \lambda_jc_j, \sum_{j=1}^r \mu_jc_j) = \prod_{j=1}^r(1- \lambda_j\overline{\mu_j})^{\frac 1 r}.\end{eqnarray}
\end{lemma}
\begin{proof} If $W$ is simple, then \cite[Lemma 5.4]{CO2} applies directly and in view of \eqref{COComparison} yields the explicit formula
\[k_W^{2r}(\sum_{j=1}^r \lambda_jc_j, \sum_{j=1}^r \mu_jc_j) = k(\sum_{j=1}^r \lambda_jc_j, \sum_{j=1}^r \mu_jc_j) = \prod_{j=1}^r(1-\lambda_j\overline{\mu_j})^2.\]
We deduce that 
\[k_W(\sum_{j=1}^r \lambda_jc_j, \sum_{j=1}^r \mu_jc_j)= \prod_{j=1}^r(1-\lambda_j\overline{\mu_j})^{\frac 1 r}.\]
in the simple case. For the general case, consider a decomposition $W = W_1 \oplus \dots
\oplus W_n$ into simple ideals. Let $r_l := {\rm rk}(W_l)$ and $(c_{l1}, \dots, c_{lr_l})$ be a Jordan frame for $W_l$. Then $(c_{11},\dots, c_{nr_n})$ is a Jordan frame for $W$
and, in fact, any Jordan frame for $W$ is of this form (as follows e.g. from \cite[Prop. X.3.2]{FK}). Let
\[z := \sum_{l=1}^n \sum_{j=1}^{r_l}\lambda_{lj}c_{lj}, \quad w:= \sum_{l=1}^n \sum_{j=1}^{r_l}\mu_{lj}c_{lj}.\]
By definition we have
\[k_W(z,w)^{{\rk W}} = \prod_{l=1}^n (k_{W_l}(z_l, w_l))^{\rk W_l},\]
where
\[z_l = \sum_{j=1}^{r_l}\lambda_{lj}c_{jl}, \quad w_l := \sum_{j=1}^{r_l}\mu_{lj}c_{lj}.\]
By the simple case we have
\[k_{W_l}(z_l, w_l)^{\rk W_l} = \prod_{j=1}^{r_l}(1- \lambda_{lj}\overline{\mu_j}),\]
and thus
\[k_W(z,w)^{{\rk W}} = \prod_{l=1}^n \prod_{j=1}^{r_l}(1- \lambda_{lj}\overline{\mu_j}).\]
\end{proof}
From this the general case follows easily:
\begin{proof}[Proof of Proposition \ref{BalancingTheorem}] If $\alpha$ is balanced, then $r_V := {\rm rk}(V)$ and $r_W := {\rm rk}(W)$ are related by $r_W = m_{\alpha}r_V$ for some
constant multiplicity $m_\alpha$. Given a
Jordan frame $(c_1, \dots, c_r)$ in $V$ and elements
\[v_1 = \sum_{j=1}^r \lambda_jc_j \in \mathcal D_{V}, \quad v_2 = \sum_{j=1}^r \mu_jc_j\in \mathcal D_{V}\]
with $\lambda_j \in \mathbb D$, $\mu_j \in \mathbb D$ we have
\[\alpha^\C(v_1) = \sum_{j=1}^r \lambda_j\alpha(c_j), \quad \alpha^\C(v_2)= \sum_{j=1}^r \mu_j\alpha(c_j).\]
Now each $\alpha(c_j)$ decomposes as
\[\alpha(c_j) = d_{j1} + \dots + d_{j\mu_{\alpha}},\]
where the $d_{jl}$ are primitive idempotents. Now we obtain
\[k_{V}(v_1,v_2)^{r_V} = \prod_{j=1}^r(1- \lambda_j\overline{\mu_j}),\]
whence
\[k_{V}(v_1,v_2)^{r_W} = \left(\prod_{j=1}^r(1- \lambda_j\overline{\mu_j})\right)^{m_{\alpha}} = \prod_{j=1}^r(1- \lambda_j\overline{\mu_j})^{m_{\alpha}}.\]
Similarly,
\[k_W(\alpha^\C(v_1),\alpha^\C(v_2))^{r_W} = \prod_{j=1}^r \prod_{l=1}^{m_\alpha}(1- \lambda_j\overline{\mu_j})=\prod_{j=1}^r(1- \lambda_j\overline{\mu_j})^{m_{\alpha}}.\]
As $k_{V}(0,0) = k_W(\alpha^\C(0),\alpha^\C(0))^{r_W}$, this implies \eqref{BalancingDream}. On the other hand, if $\alpha$ is not balanced and $v_1, v_2$ are as above, then the multiplicity function $m_\alpha$ is non-constant and thus
\begin{eqnarray*}
k_{V}(v_1,v_2) = \prod_{j=1}^r(1- \lambda_j\overline{\mu_j})^{\frac 1 r_V}
&\neq&\prod_{j=1}^r(1- \lambda_j\overline{\mu_j})^{\frac{m_{\alpha}(c_j)}{r_W}}
= k_W(\alpha^\C(v_1),\alpha^\C(v_2)).
\end{eqnarray*}
\end{proof}
\begin{example}\label{ExamplesBalanced}
The following are examples of balanced Jordan algebra homomorphisms (balanced morphisms of bounded symmetric domains):
\begin{itemize}
\item Jordan algebra homomorphisms $\alpha: V \to W$ between simple Jordan algebras (tight holomorphic morphisms between irreducible bounded symmetric domains) are balanced by Lemma
\ref{UniversalitySimpleCase}.
\item If ${\rm rk}(V) = {\rm rk}(W)$ then every injective Jordan algebra homomorphism $\alpha: V \to W$ is balanced. (Similarly for domains of equal rank.)
\item In particular, maximal polydisc embeddings are balanced.
\item Any Jordan algebra homomorphism $\alpha: \R \to W$ (any tight holomorphic disc) is balanced.
\item Compositions of balanced Jordan algebra homomorphisms (or balanced tight holomorphic morphisms) are balanced.
\end{itemize}
\end{example}
\subsection{Functoriality}
Now we can finally prove Proposition \ref{PropFunctoriality}:
\begin{proof}[Proof of Proposition \ref{PropFunctoriality}] In view of Proposition \ref{tightness} we may assume that $\mathcal D_1 = \mathcal D_{V}$ and $\mathcal D_2 = \mathcal D_W$ for Euclidean Jordan algebras $V, W$ and $\bar \beta = \alpha^\C|_{\check S_1}$ for a balanced morphism $\alpha: V \to W$. By Lemma \ref{GroupLift} we then find a finite covering group $\widehat{G}_V$ of $G_V$ and a group homomorphism $\alpha^\dagger: \widehat{G}_V \to G_W$ making the map $\alpha^\C: \mathcal D_V \to \mathcal D_W$ equivariant. In particular, given $g \in \widehat{G}_V$ there exists $h \in G_W$ such
that for all $v \in \mathcal D_V$
\begin{eqnarray}\label{Equivariance1}\alpha^\C(gv) = h\alpha^\C(v).\end{eqnarray}
By continuity, this identity also holds for all $v \in \check S_1$.
Since the actions of $\widehat{G}_V $ factors through the actions of $G_V$, we see that for every $g \in G_V$ there exists $h \in G_W$ such that
\eqref{Equivariance1} holds for all $v \in \check S_1$. Now if $(v_1, \dots, v_4)$ is extremal then by Proposition \ref{Quadruples} we
find $g \in G_V$ such that $gv_1, \dots, gv_4$ are diagonalized by a common Jordan frame $(c_1, \dots, c_r)$. Let $h \in G_W$ be an element such that \eqref{Equivariance1} holds
for all $v \in \check S_1$. Using Proposition \ref{AutInvariance} and Proposition \ref{BalancingTheorem} we now obtain
\begin{eqnarray*}
B_{\check S_1}(v_1, \dots, v_4) &=& B_{\check S_1}(gv_1, \dots, gv_4)\\
&=& \frac{k_V(gv_4,gv_1)k_V(gv_2,gv_3)}{k_V(gv_4,gv_3)k_V(gv_2,gv_1)}\\
&=& \frac{k_W(\alpha^\C(gv_4),\alpha^\C(gv_1))k_W(\alpha^\C(gv_2),\alpha^\C(gv_3))}{k_W(\alpha^\C(gv_4),\alpha^\C(gv_3))k_W(\alpha^\C(gv_2),\alpha^\C(gv_1))}\\
&=& \frac{k_W(h\alpha^\C(v_4),h\alpha^\C(v_1))k_W(h\alpha^\C(v_2),h\alpha^\C(v_3))}{k_W(h\alpha^\C(v_4),h\alpha^\C(v_3))k_W(h\alpha^\C(v_2),h\alpha^\C(v_1))}\\
&=&B_{\check S_2}(h\alpha^\C(v_1), \dots, h\alpha^\C(v_4))\\
&=& B_{\check S_2}(\bar \beta(v_1), \dots, \bar \beta(v_4)).
\end{eqnarray*}
\end{proof}

\section{Further properties of generalized cross ratios}\label{SecPropCR}

In this section we discuss a couple of basic properties of our generalized cross ratios. In the first subsection, we establish various cocycles properties. In the second subsection, we provide a way to compute generalized cross ratios; as a byproduct, we see that our generalized cross ratios are actually real-valued. Finally, we prove the axiomatic characterization of generalized cross ratios promised in Theorem \ref{Main} of the introduction.

\subsection{Cocycle properties} Generalized cross ratios satisfy various cocycle properties. The key observation for the proof of this fact is the
following simple lemma:
\begin{lemma}\label{CocycLemma}
If $X$ is a set and $k: X^2 \to \C^\times$ is an arbitrary function
then
\[b:\begin{cases}
    X^4\rightarrow  \C^\times\\
    (a,b,c,d) \mapsto\frac{k(d,a)k(b,c)}{k(d,c)k(b,a)}
   \end{cases}
\]
has the following properties:
\begin{eqnarray}
      b(a,b,c,d)&=&b(c,d,a,b) \label{b1}\\
      b(a,b,c,d)&=&b(a,b,c,x)b(a,x,c,d) \label{b2}\\
      b(a,b,c,d)&=&b(a,b,x,d)b(x,b,c,d) \label{b3}
\end{eqnarray}
\end{lemma}
\begin{proof} Straightforward computation.
\end{proof}
Since the normalized kernel is only partially defined, this does not directly apply. Still we have:
\begin{corollary}\label{BVCocyc} Let $\mathcal D$ be a bounded symmetric domain of tube type with Shilov boundary $\check S$. Then the normalized cross ratio
$B_{\check S}: \check S^{4+} \to \C^\times$ satisfies \eqref{b1}-\eqref{b3} above, whenever both sides of the equation are well-defined.
\end{corollary}
\begin{proof} Using Corollary \ref{kVProjection} we can reduce to the irreducible case. In this case, 
Lemma \ref{CocycLemma} yields \eqref{b1}-\eqref{b3} for the weighted Bergman cross ratio $B_{\mathcal D}$, and by continuity these properties extend to $B_{\check S}$.
\end{proof}

\subsection{Real values} 
Let $\mathcal D$ be a bounded symmetric domain of tube type with Shilov boudary $\check S$ and $B_{\check S}$ the generalized cross ratio of $\check S$. The goal of this subsection is to prove that $B_{\check S}$ takes values in $\R \setminus \{0,1\}$. For the computation we may assume $\mathcal D= \mathcal D_V$ for a Euclidean Jordan algebra $V$. Now let $(a,b,c,d) \in \check S^{(4+)}$; if $(a,b,c)$ is maximal then we may apply
Proposition \ref{Quadruples} in order to find $g \in
G$ and a Jordan frame $(c_1, \dots, c_r)$ of $V$ such that
\[g.(a,b,c,d) = (-e,-ie,e, \sum_{j=1}^r\lambda_j c_j).\]
Since the embedding of a maximal polydisc is balanced (Example \ref{ExamplesBalanced}), we can apply Proposition \ref{PropFunctoriality} to obtain
\[B_{\check S}(a,b,c,d) = B_{(S^1)^r}(-e,-ie,e, \lambda),\]
where $\lambda = (\lambda_j)$. Similary if $(a,b,c)$ is minimal then we find $\lambda \in  (S^1)^r$ with 
\[B_{\check S}(a,b,c,d) = B_{(S^1)^r}(e,-ie,-e, \lambda).\]
In any case we may assume $V = \R^r$, $\mathcal D= \mathbb D^r$ and $\check S = (S^1)^r$ and either $(a,b,c) = (-e,-ie, e)$ or $(a,b,c) = (e,-ie,-e)$. We will only discuss the first case here, leaving the second (completely analogous) case to the reader. Since $(-e,-ie, e, \lambda)$ is assumed extremal, the possible values of $\lambda$ are seriously restricted: Indeed, $(-1, \lambda_j, 1)$ is positive iff $\lambda_j$ is contained in the lower half-circle and negative, iff $\lambda_j$ is contained in the upper half-circle.
Since $(-e, \lambda, e)$ is either maximal or minimal we see that either $\lambda_j$ is contained in the lower half-circle for all $j = 1, \dots, r$ or in the upper half-circle
for all $j = 1, \dots, r$. Correspondingly, let us call $\lambda$ \emph{positive} or \emph{negative}. In the positive case, all the $\lambda_j$ are contained in a fixed quarter
circle. For special values of $\lambda$, the expression $B_{(S^1)^r}(-e, -ie,e, \lambda)$ is easy to compute:
\begin{lemma}\label{DullRoot}
If $\lambda_1 = \dots = \lambda_r$, then
\[B_{(S^1)^r}(-e, -ie,e, \lambda) = [-1:-i:1:\lambda_1].\]
\end{lemma}
\begin{proof} The Jordan algebra homomorphism $\R \to \R^r$ given by diagonal embedding is tight and balanced; its complexification maps $(-1, -i, 1, \lambda_1)$ to $(-e, -ie,e, \lambda)$.
Then the lemma follows from Proposition \ref{PropFunctoriality} and Example \ref{CRClassic}.
\end{proof}
This is enough information to determine the sign of $B_{(S^1)^r}(-e, -ie,e, \lambda)$ in general:
\begin{proposition}\label{PropRealCR}
The cross-ratio $B_{(S^1)^r}$ is real-valued on $((S^1)^r)^{(4+)}$. More precisely, $B_{(S^1)^r}(-e, -ie,e, \lambda)$ is positive/negative iff $\lambda$ is positive/negative.
\end{proposition}
\begin{proof} Consider the function $f: (S^1 \setminus\{-1,-i,1\})^r \to S^1$ given by
\[f(\lambda) :=\frac{B_{(S^1)^r}(-e, -ie,e, \lambda)}{|B_{(S^1)^r}(-e, -ie,e, \lambda)|}.\]
We have $B_{(S^1)^r}(-e, -ie,e, \lambda)^r=\prod[-1,-i,1,\lambda_j]\in \R$, hence  $f(\lambda)^r \in \R \cap S^1 = \{\pm 1\}$. Therefore $f$ takes values in the set $R_{2r}$ of $2r$-th roots of unity. Since $R_{2r}$ is discrete and $f$ is continuous, $f$
must be locally constant. In particular, if $\lambda$ and $\mu$ are contained in the same connected component of $(S^1 \setminus\{-1,-i,1\})^r$ and $B_{(S^1)^r}(-e, -ie,e, \mu)$ is a positive/negative
real number, then the same is true for $B_{(S^1)^r}(-e, -ie,e, \lambda)$. Combining this with Lemma \ref{DullRoot} we obtain the proposition.
\end{proof}
We can use the proposition to derive an explicit formula for the generalized cross ratio on the polydisc. Let us call an extremal
quadruple $(a,b,c,d)$ \emph{positive/negative} if it is conjugate to $(-e, -ie,e, \lambda)$ for some positive/negative $\lambda$. Then Proposition \ref{PropRealCR} and Example \ref{CRClassic} combine to the following formula:
\begin{corollary}\label{TrueRoot}
Suppose $(a,b,c)$ is maximal and $(a,b,c,d) \in ((S^1)^r)^{(4+)}$. Then
\[B_{(S^1)^r}(a,b,c,d) = \epsilon(a,b,c,d) \cdot \sqrt[r]{\left|\prod_{j=1}^r [a_j:b_j:
c_j: d_j]\right|},\]
where
\[\epsilon(a,b,c,d) = \left\{\begin{array}{ll}+1 & (a,b,c,d)\text{ positive}\\ -1 & (a,b,c,d)\text{ negative}\end{array}\right.\]
\end{corollary}
In the case, where $(a,b,c,d)$ is positive, there are two possibilities for $d$: Either, each $d_j$ lies in between $a_j$ and $b_j$ or between $b_j$ and $c_j$. This corresponds to
the cases of $(a,d,b)$ or $(b,d,c)$ being maximal. These two cases can be distinguished by the cross ratio as follows:
\begin{lemma}
If $(a,b,c)$ and $(a,d,b)$ are maximal, then $0 < B_{(S^1)^r}(a,b,c,d) < 1$. If $(a,b,c)$ and $(b,d,c)$ are maximal, then $B_{(S^1)^r}(a,b,c,d) > 1$.
\end{lemma}
\begin{proof} The assumptions imply $0 < [a_j: b_j: c_j: d_j] < 1$, respectively $[a_j: b_j: c_j: d_j] >1$ for each $j$, hence the lemma follows from the explicit formula in
Corollary \ref{TrueRoot}.
\end{proof}
We leave it to the reader to formulate the corresponding statements for the case where $(a,b,c)$ in minimal. In any case we obtain:
\begin{corollary}\label{RangeofCR}
We have $B_{\check S}(\check S^{(4+)}) = \R \setminus \{0,1\}$.
\end{corollary}
\begin{proof} Let $(a,b,c,d) \in \check S^{(4+)}$. If $(a,b,c)$ is maximal, then depending on $d$ we have either $B_{(S^1)^r}(a,b,c,d) < 0$ (if $(a,b,c,d)$ is negative) or
$B_{(S^1)^r}(a,b,c,d) < 1$ (if $(a,d,b)$ is maximal) or $B_{(S^1)^r}(a,b,c,d) > 1$ (if $(b,d,c)$ is maximal). If $(a,b,c)$ is minimal one may argue similarly (or reduce to the former
case by means of suitable cocycle properties). This shows the inclusion $\subset$. For the converse inclusion, it suffices to see that $B_{S^1}$ is onto $\R \setminus
\{0,1\}$ and $\R$ has a balanced embedding into every Euclidean Jordan algebra.
\end{proof}
As a consequence of the real-valuedness of the weighted cross ratio we obtain the following additional identity:
\begin{corollary}\label{StrangeCocycle}
For all $(a,b,c,d) \in \check S^{(4+)}$ we have
\[B_{\check S}(a,b,c,d) = B_{\check S}(b,a,d,c).\]
\end{corollary}
\begin{proof}
  This follows immediately from the real-valuedness and the property $\overline{k_{\det}(z,w)} = k_{\det}(w,z)$.
\end{proof}

\subsection{Proof of the functorial characterization}\label{MainThmProof} We claim that the family of normalized cross ratios $\{B_{\check S}\}$ satisfies Properties (i)-(iv) from Theorem \ref{Main} and is uniquely characterized by these properties. Indeed, Properties (i) and (iii) were proved in Proposition \ref{AutInvariance}, Property (ii) was established in Proposition \ref{PropFunctoriality}, and Property (iv) was checked in Example \ref{CRClassic}. It thus remains to establish uniqueness in order to prove Theorem \ref{Main}. For this we argue as follows: Given a Shilov boundary $\check S$, any $(a,b,c,d)
\in \check S^{(4+)}$ is contained in the boundary of a maximal polydisc by Proposition \ref{Quadruples}. Since the embedding of a maximal polydisc is balanced, the family
$\{B_{\check S}\}$ is uniquely determined by the family $\{B_{(S^1)^r}\}$. Condition (iii) of Theorem \ref{Main} implies that
\[B_{(S^1)^r}(a,b,c,d)^r = \prod B_{S^1}(a_j, b_j, c_j, d_j).\]
Since $B_\R$ is determined by (iv), this determines $B_{\R^r}^r$ for every $r$. Since $B_{\R^r}$ is assumed real-valued, we have in fact determined $B_{\R^r}$ up to a locally constant function into $\{\pm 1\}$. To fix this sign, consider a diagonal disc embedding $\R \to \R^r$; the transversal quadruples of the Shilov boundary $S^1$ hit every connected component, and therefore determine the sign uniquely. This shows uniqueness and finishes the proof of Theorem \ref{Main}.

\section{Translation lengths and periods}\label{SecTL}
In this section we discuss the relation between translation lengths as special isometries of bounded symmetric domains of tube type and the associated periods, which we think of as
translation lengths measured at infinity. The first subsection estimates the translation length of an arbitrary invertible linear endomorphism of a vector space $V$ on the symmetric space of $GL(V)$ in terms of the corresponding eigenvalues. The same estimate is still valid on those orbits of reductive subgroups of $GL(V)$, which are totally geodesic submanifolds. While the latter condition is automatic for semisimple subgroups, it requires some work to establish this property for the linear automorphism group $G(\Omega)$ of a cone $\Omega \subset V$. Once this is achieved, it is rather easy to estimate translation lengths for elements of $G(\Omega)$ acting on the tube over $\Omega$ with respect to the Bergman metric. The latter estimate will finally enable us to estimate translation lengths of special isometries of bounded symmetric domains of tube type via Cayley transform. Once this estimate is established, it remains only to identify the lower bound as some period of our generalized cross ratio. For computational reasons we first carry out this program in the case of irreducible bounded symmetric domains, but the passage to general bounded symmetric domains is easy due to the axioms satisfied by our generalized cross ratios.

\subsection{Translation length for linear groups}\label{SubsecTLLinear}
We recall from the introduction that given for every action of a group $G$ on a metric space $X$ the \emph{translation length} $\tau_X(g)$ of $g \in G$ on $X$ is defined by the formula\newnot{tau}
\begin{eqnarray}\tau_X(g) := \inf_{x \in X} d(x, g.x).\end{eqnarray}
If $g \in GL(V)$ is an element of the general linear group of some finite-dimensional Hilbert space $V$ and $X = \mathcal P(V)$ is given by the space of positive definite symmetric endomorphisms of $V$ (as described e.g. in \cite[Ch. II.10]{BrHa}) this translation length can be estimated easily. Since $\tau_X(g) = \tau_X(g^{-1})$ we may assume $\det(g) \geq 1$. Then we have:
\begin{lemma}\label{PVEstimate}
Let $g \in GL(V)$ and assume $\det(g) \geq 1$. Then
\[\tau_{\mathcal P(V)}(g) \geq \frac{1}{\sqrt{\dim V} } \cdot \log \det(g)^2.\]
If all eigenvalues of $g$ are of modulus $\geq 1$, then 
\[\tau_{\mathcal P(V)}(g) \leq 2 \cdot \log \det(g)^2.\]
\end{lemma}
\begin{proof} Let $p \in \mathcal P(V)$ and $c: [0, d(p, gp)] \to \mathcal P(V)$ a unit
speed geodesic joining $p$ with $gp$. We deduce from the description in \cite[Ch. II.10]{BrHa} that there exists $h \in GL(V)$ such that $p = hh^{\top}$ and a symmetric endomorphism
$X$ of $V$ of norm $1$ such that
$c(t) = h\exp(tX)h^{\top}$. Moreover, $gp= ghh^\top g^\top$. Since
$c(d(p,gp))=gp$ we have
\begin{eqnarray*}
&&h\exp(d(p, gp) \cdot X)h^{\top} = ghh^\top g^\top\\
&\Rightarrow& \det(h\exp(d(p, gp) \cdot X)h^{\top}) = \det(ghh^\top g^\top)\\
&\Rightarrow& \exp(d(p, gp) \cdot\tr(X)) = \det(g)^2\\
&\Rightarrow& \exp(d(p, gp) \cdot\tr(X)) = \exp(\log \det(g)^2)\\
\end{eqnarray*}
Since both $d(p, gp) \cdot \tr(X)$ and $\log \det(g)^2$ are real this implies
\[d(p, gp) \cdot\tr(X) = \log \det(g)^2.\]
Since $\det(g) \geq 1$ this means
\begin{eqnarray}\label{LengthInPd} d(p, gp) \cdot |\tr(X)| = \log \det(g)^2.\end{eqnarray}
Now observe that
\[|\tr(X)| = |(X|{\bf 1})| \leq \|X\|\cdot\|{\bf 1}\| = 1 \cdot \sqrt{\dim V} = \sqrt{\dim V}.\]
Inserting into \eqref{LengthInPd} we obtain
\[d(p, gp) \geq \frac{1}{\sqrt{\dim V} }|\log \det(g)^2|.\]
Passing to the infimum over all $p \in \mathcal P(V)$ we obtain the first inequality.\\

For the converse inequality we use the following consequence of the existence of a real Jordan canonical form:  Assume that the eigenvalues of $g$ (with multiplicity) are given by $\lambda_1, \dots, \lambda_m$. Then there exists a sequence $h_n \in GL(V)$ such that $(h_n^{-1}gh_n)(h_n^{-1}gh_n)^\top$ converges to a diagonal matrix $\hat g$ with entries $|\lambda_1|^2, \dots |\lambda_m|^2$. In particular we obtain
\begin{eqnarray*}
\tau(g) &=& \inf_{h \in GL(V)} d(hh^\top, ghh^\top g^\top) \leq d({\rm Id}_V, \hat g).
\end{eqnarray*}
Then \cite[Cor. 10. 42]{BrHa} yields 
\begin{eqnarray*}
\tau(g) &\leq& \left(\sum_{j=1}^m (\log|\lambda_j|^2)^2 \right)^{\frac 1 2}\\
&\leq  & 2 \cdot \sum_{j=1}^m |\log|\lambda_j||.
\end{eqnarray*}
Now, if $|\lambda_j| > 1$ for all $j=1, \dots, m$, then the right hand side is precisely given by $2 \cdot \log \det(g)^2$.
\end{proof}
We can use the lemma to compute translation lengths for isometry groups of totally geodesic subspaces of ${\mathcal P(V)}$ by means of the following general result:
\begin{lemma}\label{TLSubspace}
 Let $X$ be a complete $CAT(0)$-manifold. Let $Y\subset X$ a totally geodesic subspace and $h$ be an isometry of $X$ with $hY\subset Y$. Then
  \[\tau_X(h)=\tau_Y(h).\]  
\end{lemma}
\begin{proof} By assumption, $Y$ is closed, convex and complete with respect to the induced metric. This implies \cite[II.2.4]{BrHa} that there exists an orthogonal projection $\pi: X \to Y$. Given $x \in X \setminus Y$ we denote by $\sigma_x$ the constant speed geodesic with $\sigma_x(0) = \pi(x)$, $\sigma_x(1) = x$. By construction, $\sigma_x$ is the unique geodesic which contains $x$ and intersects $Y$ orthogonally. This description implies in particular that
\begin{eqnarray}\label{hSigma}
h\sigma_x = \sigma_{hx} \quad (x \in X \setminus Y).
\end{eqnarray}
For any $y \in Y$ denote by $\tau_y$ the geodesic joining $y$ and $h.y$. By assumption, $\tau_{y}$ is contained in $Y$ for every $y \in Y$. In particular, given $x \in X \setminus Y$, the geodesic $\tau_{\pi(x)}$ is orthogonal to both $\sigma_x$ and $h.\sigma_{x}$, whence the shortest connetion between these two geodesics. We deduce that
\begin{eqnarray*}
d(\sigma_x, h.\sigma_x) = d(\sigma_x \cap \tau_{\pi(x)}, h.\sigma_x\cap \tau_{\pi(x)}) = d(\sigma_x(0), h.\sigma_x(0)). 
\end{eqnarray*}
Combining this with \eqref{hSigma} we obtain for all $x \in X \setminus Y$ the inequality
\[d(x,hx) = d(\sigma_x(1), \sigma_{hx}(1)) \geq d(\sigma_x, h.\sigma_x) = d(\sigma_x(0), h.\sigma_x(0)) = d(p(x), h.p(x)).\]
Then the lemma follows by passing to the infimum.
\end{proof}
We will now combine these two observations to estimate the translation lengths of certain special isometries on bounded symmetric domains. For this we return to our previous notation, i.e. $\mathcal D$ is a bounded symmetric domain realized by means of a Euclidean Jordan algebra $V$ and $G = G_V$. All computations in the remainder of this section are with respect to the (unnormalized) Bergman metric on $\mathcal D \subset V$. In particular, all translation lengths $\tau_{\mathcal D}$ will be with respect to this metric.\\

We now consider isometries $g \in G$ which admit a pair of transverse fixed points $g^{\pm} \in \check S$. We label these two fixed-points in such a way that either $g^{-}$ is non-attractive or $g^+$ is non-repellent and fix some $h \in G$ with $hg^{\pm} = \pm e$. Then 
\begin{eqnarray}
g_1 := h g h^{-1} \in L(Q^{+}),
\end{eqnarray}
and hence
\begin{eqnarray}\label{g2fromg1}
g_2 := \hat c(g_1) = c \circ g_1 \circ c^{-1} \in G(\Omega)
\end{eqnarray}
by Proposition \ref{Levi}. The elements $g_1$ and $g_2$ will depend on the choice of $h$, but the eigenvalues of $g_2$ (considered as an element of $GL(V)$) and, in particular, the determinant of $g_2$ will not. By our choice of fixed points we have $\det(g_2) \geq 1$, hence $g_2$ has at least one eigenvalue of modulus $\geq 1$. If the modulus of all eigenvalues is strictly greater than $1$, then $g^+$ is attractive and $g^-$ is repellent. We then call $(g^+, g^-)$ an \emph{attractor-repellor pair} for $g$. This will be the case in the situations we are most interested in. However, half of our estimates work also without any hyperbolicity assumptions.
\begin{proposition}\label{EstimateWeaklyHyperbolic} Assume $g \in G$ has two transverse fixed points $g^\pm$ labelled as above. Then
\[ \tau_{\mathcal D}(g) \geq \frac{1}{2\cdot \sqrt{\dim V}} \cdot \log \det(g_2)^2,\]
and if all eigenvalues of $g_2$ have modulus $\geq 1$, then
\[ 
\tau_{\mathcal D}(g) \leq  \log \det(g_2)^2.
\]
\end{proposition}
\begin{remark}\label{NotSharp}
The main idea in the proof of Proposition \ref{EstimateWeaklyHyperbolic} is to apply Lemma \ref{PVEstimate} to a linear representation for the Levi factor of a Shilov parabolic. We will use the representation constructed in Corollary \ref{LeviRep}. This representation has the advantage that it can be defined for all classical and exceptional bounded symmetric domains of tube type. We thus obtain a uniform proof of Proposition \ref{EstimateWeaklyHyperbolic}. The disadvantage of our choice of representation is that its dimension is in general much larger than would be necessary; consequently, the constants in Proposition \ref{EstimateWeaklyHyperbolic} are not sharp. Indeed, a case by case argument can be used to provide better constants, most notably in the symplectic case. Since the optimal constants will not be relevant for us, we will not carry out the necessary case by case considerations here.
\end{remark}
\begin{proof}[Proof of Proposition \ref{EstimateWeaklyHyperbolic}]
Since $c$ is an isometry between $\mathcal D$ and $T_\Omega$ we obtain
\begin{eqnarray}\label{Tg2fromg1}
\tau_{\mathcal D}(g) &=& \tau_{T_\Omega}(g_2).
\end{eqnarray}
In view of Lemma \ref{PVEstimate} it thus suffices to establish the equalities
\begin{eqnarray}\label{ProjectionGoal}
    \tau_{T_\Omega}(g_2)  =  \tau_\Omega(g_2)= \frac{1}{2} \tau_{\mathcal P(V)}(g_2).
\end{eqnarray}
Here we think of $\Omega$ as equipped with the restriction of the Hermitian metric from $T_\Omega$. Both equalities are actually consequences of Lemma \ref{TLSubspace}, so let us verify the assumptions: As far as the first equality is concerned, we need to show that the inclusion $i\Omega \subset T_{\Omega}$ is totally geodesic. This seems to be well-known (as stated in \cite[p. 361]{Rothaus} without proof), but for lack of reference let us work out the details: Since $T_\Omega$ is an open subset of $V^\C$, we can identify the
tangent space of $T_\Omega$ at any point $z \in T_\Omega$ with $V^\C$ using the linear connection on $V^\C$. Under this
identification, the Hermitian metric $H$ on $T_\Omega$ admits the following
description (see \cite[Prop. X.1.3]{FK}): Let $n := \dim V$, $r := {\rm rk}(V)$. Then given $z \in
T_\Omega$ and $a,b \in V^\C$ we have
\[H_z(a, b) = \left.\left(\frac{2n}{r}P\left(\frac{z-\bar z}{i}\right)^{-1}a \,\right|\, b\right) = \left.\left(\frac{2n}{r}P\left(2{\rm Im}(z)\right)^{-1}a\,\right|\,b\right) = H_{{\rm Im}(z)}(a,b).\]
In other words, translation in the direction of the real axis is
isometric for $H$. We have $H = g + i \omega$, where $g$ is the
Riemannian metric on $T_\Omega$ and $\omega$ is the K\"ahler form.
In particular, since $\omega$ is skew-symmetric, we have for all
$z \in T_\Omega$ and all $a \in V^\C$ the equality
\begin{eqnarray*}g_z(a,a) = H_z(a,a) &=& H_z({\rm Re}(a),{\rm Re}(a)) + H_z(i{\rm Im}(a),i{\rm Im}(a))\\
&=& g_z({\rm Re}(a),{\rm Re}(a)) + g_z(i{\rm Im}(a),i{\rm Im}(a)).
\end{eqnarray*}
In particular,
\[g_z(a,a) \geq g_z(i{\rm Im}(a),i{\rm Im}(a)) = g_{i{\rm Im}(z)}(i{\rm Im}(a),i{\rm Im}(a)).\]
Thus, given any path $\sigma:[0,1] \to T_\Omega$ with $\sigma(0) = z$, $\sigma(1) = hz$ we have
\begin{eqnarray*}
l(\sigma) &=& \int_{0}^1\sqrt{g_{\sigma(t)}(\dot\sigma(t),\dot\sigma(t) )}dt\\
&\geq& \int_{0}^1\sqrt{g_{i{\rm Im}(\sigma(t))}(i{\rm Im}(\dot\sigma(t)),i{\rm Im}(\dot\sigma(t)))}dt\\
&=& l(i{\rm Im}(\sigma(t))).
\end{eqnarray*}
This means that for every paths in $T_\Omega$ between two points in $i\Omega$ the projection of this path into $i\Omega$ is at most as long. This is precisely, what we had to show.\\ 

For the second equality in \eqref{ProjectionGoal} we argue as follows: Since the stabilizer of $e$ in $G(\Omega)$ is  given by $K(\Omega) := G(\Omega)\cap O(V)$ \cite[Prop. I.4.3]{FK} there is a natural embedding $\iota: \Omega \hookrightarrow \mathcal P(V)$ induced by the inclusion $\hat\iota: G(\Omega) \to GL(V)$. We would like to show that $\iota$ is totally geodesic and isometric with respect to \emph{twice} the restriction of the Bergman metric on $T_{\Omega}$ to $i\Omega$ and the natural metric on $\mathcal P(V)$ used earlier. (This will account for the addition factor $\frac 1 2$.) The second statement is again a simple computation: Denote by $I \in \mathcal P(V)$ the identity matrix. Under the canonical identifications $T_e\Omega \cong V$ and $T_I\mathcal P(V) = {\rm Sym}_{\dim V}(\R)$ the differential of the embedding $\iota$ at $e$  is given by \cite[Thm. III.3.1]{FK}
\[d\iota_e: V \to {\rm Sym}_{\dim V}(\R), \quad x \mapsto L(x).\]
The Bergman metric in $x$ is given by the formula \cite[X.1.3 and Ch. III.4]{FK}
\[
  H_x(u,v):=\frac{r}{2n}\tr_V((P(x)^{-1}u)v),
\]
where $r := \rk V$, $n := \dim V$ and $\tr_V$ is again the Jordan algebra trace. On the other hand, the metric in $\mathcal P (V)$ is given by \cite[Ch. II.10]{BrHa}
\[
  g_x(X,Y)=\tr(x^{-1}Xx^{-1}Y),
\]
where $\tr$ is the usual matrix trace. Since $\tr_V(x)=\frac{r}{n}\cdot \tr(L(x))$ \cite[III.4.2]{FK} we have 
\[
  H_e(u,v)=\frac{1}{2}g_I(L(u),L(v)).
\]
Both the Bergman metric and the restriction of the metric on $\mathcal P(V)$ to the image of $\Omega$ are invariant under $G(\Omega)$; we thus deduce that the Riemannian metrics on $\Omega$ and $\iota(\Omega)$ coincide up to a global factor of $\frac 1 2$ as claimed. We are thus left with proving that $\iota(\Omega)$ is totally geodesic in $\mathcal P(V)$.\\ 

For this we observe that $\iota(\Omega)$ is the orbit of the reductive subgroup $G(\Omega) < GL(V)$. By \cite[Thm. II.10.58]{BrHa} the orbit of such a subgroup $G$ is totally geodesic if for all $X \in \L{gl}(V)$ with $\exp(X) \in G$ already $\exp(tX) \in G$ for all $t \in \R$. Let us verify this for 
$G = G(\Omega)$:  Let $K(\Omega) := G(\Omega) \cap O(V)$ and denote by $\L{p}(\Omega)$ the symmetric matrices in the Lie algebra $\L g(\Omega) \subset  \L{gl}(V)$ of $G(\Omega)$. Then $G(\Omega)$ admits a polar decomposition $G(\Omega) = K(\Omega)\exp(\L{p}(\Omega))$ \cite[Prop. I.1.9, I.4.3 and Thm. III.5.1]{FK}. In particular, if $X$ is a symmetric matrix with $\exp(X) \in G$, then there exist $k \in K(\Omega)$ and $Y \in \L p(\Omega)$ such that
\[e^X = ke^Y \Rightarrow e^{2X} = (e^X)^\top e^X = (ke^Y)^\top ke^Y = e^{2Y}.\]
Then the uniqueness of the Polar decomposition in $GL(V)$ yields $2X = 2Y$, whence $X \in \L p(\Omega)$. This shows that $\iota(\Omega)$ is totally geodesic in $\mathcal P(V)$ and finishes the proof.
\end{proof}

\subsection{Comparison to periods of generalized cross ratios} We keep the notation of the last subsection, in particular $g$ denotes an isometry of $\mathcal D$ with transverse fixed points $g^\pm$ labelled as before. Let us call $z \in \check S$ \emph{admissible} if $(g^-, z, g^+, gz) \in \check S^{(4)}$. Given and admissible point $z$ we define 
the \emph{period of $(g, g^+, g^-)$ with respect to $z$} \newnot{tauinfty} by
\[\tau_\Dca^\infty(g, g^+, g^-)_z := \log B_{\check S}(g^-, z, g^+, gz).\]
If $g^+$ and $g^-$ are clear from the context, we write $\tau_\Dca^\infty (g)$.
\begin{lemma}
The above period does not depend on the admissible point used to define it, i.e.
\[\tau_\Dca^\infty(g, g^+, g^-) := \tau_\Dca^\infty(g, g^+, g^-)_z\]
is well-defined. 
\end{lemma}
\begin{proof} Let $F(y) := B_{\check S}(g^-, y, g^+, gy)$. We claim that $F$ is constant on the set
\[X' := \{w \in \check S\,|\, (g^-, w, g^+, gw) \in \check S^{(4)}\} \subset \check S.\]
If $(g^-, y, g^+, z) \in \check S^{(4)}$ then 
\begin{eqnarray*}
F(z) = B_{\check S}(g^-, z, g^+, gz) &=& B_{\check S}(g^-, z, g^+, y) \cdot  B_{\check S}(g^-, y, g^+, gz)\\
&=& B_{\check S}(gg^-, gz, gg^+, gy) \cdot  B_{\check S}(g^-, y, g^+, gz)\\
&=& B_{\check S}(g^-, y, g^+, gz) \cdot B_{\check S}(g^-, gz, g^+, gy)\\
&=& B_{\check S}(g^-, y, g^+, gy) = F(y);
\end{eqnarray*}
otherwise we can find $w \in X'$ with $(g^-, y, g^+, w), (g^-, z, g^+, w) \in \check S^{(4)}$ which then yields $F(y) = F(w) = F(z)$.
\end{proof}
\begin{remark}
Strictly speaking, our axiomatically defined generalized cross ratio has domain $\check S^{(4+)}$, so that the period can only be defined if $(g^-, z, g^+, gz) \in \check S^{(4+)}$. However, we have constructed an explicit model of the cross ratio on all of $\check S^{(4)}$, which on the subset $ \check S^{(4+)}$ agrees with the axiomatic one. The last lemma then implies that the period as defined above only depends on the axiomatically defined cross ratio, but in order to compute it we can use our concrete model as defined on all of $\check S^{(4)}$.
\end{remark}
Now we can state the main result of this section; our first formulation is for irreducible bounded domains:
\begin{theorem}\label{TranslationLengthFinal} Let $\mathcal D$ be an irreducible bounded symmetric domain of tube type and $g \in G=G(\mathcal D)$ with two transverse fixed points $g^{\pm}$ labelled as above. Then
\[
 \tau_\Dca(g) \geq {\sqrt{\dim_\C \mathcal D}}\cdot \tau_\Dca^\infty(g, g^+, g^-),
\]
and if all eigenvalues of $g_2$ have modulus $\geq 1$, then 
\[
\tau_{\mathcal D}(g) \leq  2 \dim_\C \mathcal D \cdot \tau_\Dca^\infty(g, g^+, g^-).
\]
\end{theorem}
From this we derive the following result in the general case:
\begin{corollary}\label{CRTransl}
Let $\mathcal D$ be a bounded symmetric domain of tube type which decomposes as  $\mathcal D = \mathcal D_1 \times \dots \times \mathcal D_m$ into irreducible bounded symmetric domains. Assume that $g \in G(\mathcal D)^0$ admits two transverse fixed points $g^{\pm} \in \check S$ labeled as above. Then 
\[
  \tau_{\mathcal D}(g) \geq  \sqrt{\min_j \dim_\C \mathcal D_j} \cdot \tau_{\mathcal D}^\infty(g, g^+,     g^-),
\]
and if all eigenvalues of $g_2$ have modulus $\geq 1$, then 
\[
\tau_\Dca(g) \leq  2\cdot  \rk \mathcal D \cdot \max_j \frac{ \dim_\C \mathcal D_j}{\rk \mathcal D_j} \cdot \tau^\infty(g).
\]
\end{corollary}
\begin{proof} Identifying $G(\mathcal D)^0$ with the product of the group $G(\mathcal D_j)^0$ we can then write $g = (g_1, \dots, g_m)$ for some $g_j \in G(\mathcal D_j)^0$. Let us abbreviate $r_j := \rk \mathcal D_j$, $n_j := \dim_\C \mathcal D_j$, $r := \rk \mathcal D$, $\tau_j := \tau_{\mathcal D_j}(g_j)$ and $\tau_j^\infty := \tau_{\mathcal D_j}^\infty(g_j)$ so that
\[\tau_{\mathcal D}(\rho(\gamma)) = \sqrt{\sum_{j=1}^m\tau_j^2}, \quad \tau_{\mathcal D}^\infty(\rho(\gamma)) =  \sum_{j=1}^m\frac{r_j}{r}\tau_j^\infty.\]
By Theorem \ref{TranslationLengthFinal} we thus obtain
\begin{eqnarray*}
\tau_{\mathcal D}^\infty(g) &=&  \sum_{j=1}^m\frac{r_j}{r}\tau_j^\infty \leq \left(\sum_{j=1}^m\frac{r_j}{r}\right) \cdot \max_j \tau_j^\infty\\
&\leq& \sqrt{\sum_{j=1}^m (\tau_j^\infty)^2} \leq \max_{j} \frac{1}{\sqrt{n_j}} \cdot  \sqrt{\sum_{j=1}^m n_j (\tau_j^\infty)^2}\\
&\leq& \frac{1}{\sqrt{\min_j n_j}}  \sqrt{\sum_{j=1}^m \tau_j^2} =  \frac{1}{\sqrt{\min_j n_j}} \tau_{\mathcal D}(g).
\end{eqnarray*}
For the other inequality Theorem \ref{TranslationLengthFinal} yields
\begin{eqnarray*}
\tau_{\mathcal D}(g) &=&  \sqrt{\sum_{j=1}^m \tau_j^2 }\leq \sum_{j=1}^m \tau_j
 \leq 2 \cdot \sum_{j=1}^m \dim_\C \Dca_j \cdot \tau_j^\infty\\ &=& 2 \cdot \sum_{j=1}^m \dim_\C \Dca_j \cdot r \cdot \frac{n_j}{r_j}\cdot \frac{r_j}{r}\cdot  \tau_j^\infty \\
 &\leq & 2\cdot r \cdot \max_j \frac{n_j}{r_j} \cdot \tau^\infty.
\end{eqnarray*}
\end{proof}
Again the constants are not sharp and could be improved along the lines described in Remark \ref{NotSharp}.\\

The remainder of this section is devoted to the proof of Theorem \ref{TranslationLengthFinal}. We will use the notation $ \tau_\Dca^\infty(g)$ as a shorthand for $ \tau_\Dca^\infty(g, g^+, g^-)$. 
The following is immediate from the $G$-invariance of the generalized cross ratio:
\begin{lemma} There exists a dense open subset $X \subset \check S$ such that for all $z \in X$ we have
\[\tau_{\mathcal D}^\infty(g) = \log B_{\check S}(-e, z, e, g_1z).\]
\end{lemma}
Now we use our assumption that $\mathcal D$ is irreducible; we thus have $B_{\mathcal D} = B_{\mathcal D}^{(-\frac 1 {2\dim V})}$. We then use Proposition \ref{BergmanCrossInvarianceInterior} to deduce that
\[B_{\mathcal D}(x,y,z,t) = B_{T_\Omega}^{(-\frac{1}{2\dim V})}(c(x),c(y),c(z),c(t)).\]
where $(x,y,z,t) \in \mathcal D^4$, $T_\Omega$ is the tube over $\Omega$ and $c: \mathcal D \to T_\Omega$ is the Cayley transform. Thus, if $x_n$ is a sequence in $\mathcal D$ converging to $e$ then for all $w$ in a dense open subset of $V$ we have
\begin{eqnarray*}\tau_{\mathcal D}^\infty(g) &=& \frac{1}{2\cdot \dim V} \cdot \log\left( \frac{k_{T_\Omega}(w,0)}{k_{T_\Omega}(g_2w, 0)} \cdot \lim_{n \to \infty} \frac{k_{T_\Omega}(g_2w, c(x_n))}{k_{T_\Omega}(w, c(x_n))}\right).
\end{eqnarray*}
Now the right hand side can be computed explicitly:
\begin{proposition}\label{VTL} With the notations above we have:
\begin{eqnarray*}
  \tau_{\mathcal D}^\infty(g) &=&\frac{1}{2 \cdot {\dim V}} \cdot \log \det(g_2)^2
\end{eqnarray*}
\end{proposition}
\begin{proof} We first show that
\[\lim_{n \to \infty} \frac{k_{T_\Omega}(g_2w, c(x_n))}{k_{T_\Omega}(w, c(x_n))} = 1.\]
Indeed, let $\lambda\in [0,1)$. Then
\[
c(\lambda \cdot e )=i\frac{1+\lambda}{1-\lambda}e.
\]
Using \cite[X.1.3]{FK} we obtain
\begin{eqnarray*}
\lim_{n \to \infty} \frac{k_{T_\Omega}(g_2w,c(x_n))}{k_{T_\Omega}(w,c(x_n))} &=&\lim_{\lambda \rightarrow
1} \left(\frac{\det(g_2w-i\frac{1+\lambda}{1-\lambda}e)}{\det(w-i\frac{1+\lambda}{1-\lambda}e)}\right)^{-\frac{2n}{r}}\\
&=&
\lim_{\lambda \rightarrow
1}\left(\frac{\det(\frac{1-\lambda}{1+\lambda}g_2w-ie)}{\det(\frac{1-\lambda}{1+\lambda}w-ie)}\right)^{-\frac{2n}{r}}
=1.
\end{eqnarray*}
Now it suffices to show that $\frac{k_{T_\Omega}(w,0)}{k_{T_\Omega}(g_2w, 0)} = \det(g_2)^2$. Since $g_2: T_\Omega \to T_\Omega$ is biholomorphic, we see from \cite[Prop. IX.2.4]{FK} that
\[k_{T_\Omega}(w, 0) = k_{T_\Omega}(g_2w, g_20)\det{}_{\C}J_{g_2}(w) \overline{\det{}_{\C}J_{g_2}(0)},\]
where $J_{g_2}$ denotes the complex Jacobi matrix of $g_2$. Note that $g_2$ is a real matrix, because it is in $G(\Omega)^0\subset GL(V)$. Since it is linear, we have $J_{g_2} \equiv g_2$ and $g_20=0$, whence
\[k_{T_\Omega}(w, 0) = k_{T_\Omega}(g_2w, g_20)\det{}_{\C}J_{g_2}(w) \overline{\det{}_{\C}J_{g_2}(0)} = k_{T_\Omega}(g_2w, 0)\det(g_2)^2.\]
Dividing both sides by $k_{T_\Omega}(g_2w, 0)$ the proposition follows.
\end{proof}
Now the theorem follows easily:
\begin{proof}[Proof of Theorem \ref{TranslationLengthFinal}] Since $\dim_\C \mathcal D = \dim V$ the estimates in Proposition \ref{EstimateWeaklyHyperbolic} and Proposition \ref{VTL} yield
\begin{eqnarray*}
\tau_{\mathcal D}(g) &\geq& \frac{1}{2\cdot \sqrt{\dim V}} \log \det(g_2)^2
={\sqrt{\dim_\C \mathcal D}} \cdot \tau_{\mathcal D}^\infty(g).
\end{eqnarray*}
and
\begin{eqnarray*}
  \tau_{\mathcal D}(g) &\leq& \log \det(g_2)^2 = 2\cdot \dim_\C \Dca\cdot  \tau_{\mathcal D}^\infty(g).
\end{eqnarray*}
\end{proof}

\section{Maximal representations, strict cross ratios and well-displacing}\label{SecMaxRep}
Let $\varrho:\Gamma \rightarrow G$ a maximal representation into a Hermitian group of tube type. 
In this section we explain how the generalized cross ratio functions defined above can be used to associate with $\varrho$ a strict cross ratio on the circle in the sense of Labourie \cite{Lab05}. We then combine our estimates for the translation lengths with Labourie's equivalence theorem for strict cross ratios (or rather a version thereof, as given in Appendix \ref{AppLabourie}) in order to derive the well-displacing property for maximal representations. We then deduce the corollaries given in the introduction.
 
\subsection{Maximal representations and limit curves}

Returning to the notation of the introduction, let $\Sigma$ be a closed, oriented
surface of genus $g\geq 2$ with fundamental group $\Gamma$. We fix a hyperbolization of $\Sigma$, i.e. a faithful homomorphism $\Gamma \to PU(1,1)$ with discrete image so that
$\Sigma = \Gamma\backslash \mathbb D$. In particular, we obtain an action of $\Gamma$ on the circle.\\

We also fix a Euclidean Jordan algebra $V$ and denote by $\mathcal D$ and $\check S$ respectively the associated
bounded symmetric domain and Shilov boundary. The corresponding groups $G, K, Q_+$ are defined as before. The aim of this section is to construct a strict cross ratio on the circle in the sense of \cite{Lab05} associated with a
maximal representation $\varrho:\Gamma
\rightarrow G$\newnot{MR}. Our basic references concerning maximal representations are \cite{Surface}, \cite{Anosov} and \cite{LimitCurves}. Let us briefly
recall the main definitions: Denote by
$\omega_{\mathcal D}$ the K\"ahler form on $\mathcal D$ associated with the metric of minimal holomorphic sectional curvature $-1$. Given an arbitrary $\rho$-equivariant map $f:
\mathbb D \to \mathcal D$ we define the \emph{Toledo invariant} $T_{\varrho}$ of $\rho$ by
\[
  T_\varrho:=\frac{1}{2\pi} \int_{\Sigma}f^*\omega_\Dca.
\]
This does not depend on the choice of $f$. The Toledo invariant satisfies a generalized Milnor-Wood inequality, in the present normalization given by
\begin{eqnarray}\label{MW}|T_\varrho|\leq |\chi(\Sigma)|\cdot \rk(V).\end{eqnarray}
Accordingly, the representation $\varrho$ is called \emph{maximal} if $T_\varrho = |\chi(\Sigma)|\cdot \rk(V)$. 
\begin{definition} Let $\varrho: \Gamma \to G$ be a representation. A  $\rho$-equivariant Borel map $\phi: S^1 \to \check S$ is called a \emph{limit curve} for $\rho$. It is called \emph{monotone}, if it maps positive/negative triples in $S^1$ to triples of maximal/minimal Maslov index in $\check S$. \newnot{LC}
\end{definition}
Then we have:
\begin{theorem}[Burger-Iozzi-Wienhard]\label{BIW}
A representation is maximal iff it admits a monotone continuous limit curve.
\end{theorem}
Indeed, it was proved in \cite[Thm. 8] {Surface} that maximal representations are characterized by the existence of a left-continuous monotone limit curve $\phi$. The fact that any such curve is actually continuous was later proved in \cite{LimitCurves}. The main step in the latter proof is to show that every maximal representation has the Anosov property as defined in \cite{LabourieAnosov}. In the symplectic case this property is established in \cite{Anosov}; the general case appears in \cite{LimitCurves}. In fact, as kindly pointed out to us by Olivier Guichard, the Anosov property of maximal representations also yields the following result:
\begin{proposition}\label{LCUnique}
Every maximal representation admits a unique monotone continuous limit curve.
\end{proposition}
The deduction of this proposition from the Anosov property is essentially straight forward. Since the necessary notation involved is however rather heavy we defer the details of the proof to Appendix \ref{AppLimitCurve}. Now let $\varrho: \Gamma \to G$ be a maximal representation and $\phi: S^1 \to \check S$ the associated monotone continuous limit curve. As a consequence of monotonicity two distinct points $x\neq y\in S^1$ are mapped to transverse points under $\phi$; thus if
\[(S^1)^{4*} := \{(x,y,z,t) \in (S^1)^4\,|\, x\neq t, y\neq z\}\]
denotes the domain of the classical cross ratio, then we obtain a map
\[\varphi^{(4)}: (S^1)^{4*} \to \check S^{(2)} \times \check S^{(2)}, \quad (x,y,z,t) \mapsto (\varphi(x),\varphi(y),\varphi(z),\varphi(t)).\]
Moreover, the dense subset $(S^1)^{(4)} \subset (S^1)^{4*}$ satisfies
\begin{eqnarray}\label{TransverseToExtremal}
\varphi^{(4)}((S^1)^{(4)}) \subset \check S^{(4+)},
\end{eqnarray}
where the right hand side is precisely the domain of definition of 
$B_{\check S}$. We may thus define a function
 \begin{eqnarray*}b_\varrho := (\phi^{(4)})^*B_{\check S}: (S^1)^{(4)} \to \R \setminus\{0,1\}, \quad (x,y,z,t) \mapsto B_{\check S}(\phi(x),\phi(y),\phi(z),\phi(t)).\end{eqnarray*}
Since $B_{\check S}$ extends continuously to $\check S^{(2)} \times S^{(2)}$, we may also extend $b_\rho$ to a continuous function
  \begin{eqnarray*}b_\varrho: (S^1)^{4*} \to \R.\end{eqnarray*}
\begin{definition} \label{DefCR} Let $\varrho: \Gamma \to G$ be a maximal representation and $\varphi: S^1 \to \check S$ an associated limit curve. Then the function $b_\varrho: (S^1)^{4*} \to \R$ defined above is called the \emph{cross ratio of the maximal representation $\varrho$}. \newnot{CR}
\end{definition}
Because of our functorial construction of generalized cross ratios the following functoriality of the $b_\rho$ comes for free:
\begin{proposition} 
Let $G, H$ be Hermitian Lie groups of tube type, $\rho: \Gamma \to H$ a maximal representation and $t: H \to G$ a homomorphism inducing a tight holomorphic morphism of the underlying bounded symmetric domains. Then $t \circ \rho$ is maximal and $b_{\rho} = b_{t\circ \rho}$.
\end{proposition}
\begin{proof} The homomorphism $t$ induces a map $t_*: \check S_H \to \check S_G$ of the corresponding Shilov boundaries \cite{Tight}. Now if $\phi$ is a limit curve for $\rho$, then $t_* \circ \phi$ is a limit curve for $t \circ \rho$. Thus the proposition follows from Property (ii) of Theorem \ref{Main}.
\end{proof}
The main properties of cross ratios of maximal representations are collected in the following theorem:
\begin{theorem}\label{CRStrict}
The cross ratio $b_\varrho: (S^1)^{4*} \to \R$ is a continuous $\Gamma$-invariant function satisfying the following properties:
\begin{eqnarray}
      b_{\varrho}(x,y,z,t)&=&b_{\varrho}(z,t,x,y) \label{a1}\\
      b_{\varrho}(x,y,z,t)&=&b_{\varrho}(x,y,z,w)b_{\varrho}(x,w,z,t) \label{a2}\\
      b_{\varrho}(x,y,z,t)&=&b_{\varrho}(x,y,w,t)b_{\varrho}(w,y,z,t) \label{a3}\\
      x=z \text{ or } y=t  &\Leftrightarrow& b_{\varrho}(x,y,z,t)=1  \label{a4}\\
      t=x \text{ or } z=y &\Leftrightarrow& b_{\varrho}(x,y,z,t)=0  \label{a5}
\end{eqnarray}
\end{theorem}
\begin{proof} $\Gamma$-invariance on $(S^1)^{(4)}$ follows from $\Gamma$-equivariance of $\varphi$ and $G$-invariance of $B_{\check S}$ on $S^{(4+)}$ (Proposition \ref{AutInvariance}). By continuity we
obtain $\Gamma$-invariance on all of $(S^1)^{4*}$. By a similar extension argument, Properties \eqref{a1}-\eqref{a3} follow from Corollary \ref{BVCocyc}. For $(x,y,z,t) \in (S^1)^{(4)}$ the inclusion \eqref{TransverseToExtremal} together with Proposition \ref{RangeofCR} implies $b_{\varrho}(x,y,z,t)\not\in \{0,1\}$. It thus remains to consider the cases $x=z$, $y=t$, $t=x$ and $z=y$. In the last two cases 
the vanishing of $\varphi^*k_V$ along the diagonal implies $b_{\varrho}(x,y,z,t)=0$. In the first two cases we get $b_{\varrho}(x,y,z,t)=1$ as a consequence of the similar property for $B_{\check S}$. This establishes \eqref{a4}-\eqref{a5} and finishes the proof.
\end{proof}
In fact, it follows from Corollary \ref{StrangeCocycle} that $b_\varrho$ also satisfies
\[b_\varrho(x,y,z,t) = b_\varrho(y,x,t,z).\]
However, we are not going to use this property in the sequel. In the language of \cite{Lab05} the
theorem says precisely that $b_\varrho$ is a \emph{strict cross ratio}. Concerning such cross ratios we have the following equivalence theorem of Labourie:
\begin{theorem*}[Labourie]\
  Let $b_1$ and $b_2$ be strict cross ratios. Then there exists constants $C,D>0$ such that
  \[D^{-1}|\log b_1|-1\leq |\log b_2|\leq C|\log b_1|+C\]
\end{theorem*}
This is implicitly contained in \cite{Lab05}. For the convenience of the reader we provide a self-contained proof in Appendix \ref{AppLabourie}, see Theorem \ref{PropLab}. Here we just need the following corollary:
\begin{corollary}\label{CREq}
Let $\rho: \Gamma \to G$ be a maximal representation with associated cross ratio $b_\rho$. Then there exists $C>0$ such that for all $(x,y,z,t) \in (S^1)^{4*}$,
\[|\log b_\rho(x,y,z,t)| \geq C \cdot |\log[x:y:z:t]|-1.\]
\end{corollary}

For maximal representations into symplectic groups there is a more classical construction of an associated cross ratio. It is also of the form 
\[b_{class}(x,y,z,t):= B_{class}(\phi(x),\phi(y),\phi(z),\phi(t)),\]
but now $B_{class}$ is the classical symplectic cross ratio as described e.g. in \cite[Ch. 3.2.5]{Lab05}. It is instructive to compare our construction to the classical one:
\begin{example}\label{ExSympl} We claim that for $G = Sp(2n,\R)$ our cross ratio $B_{\check S}$ provides a specific $n$th root for the classical cross ratio $B_{class}$ on its domain of definition.
The claimed relation between the two cross ratios can be established by a direct computation. It clearly suffices to compare the two cross ratios on the Shilov boundary $\check S_P$ of a given maximal polydisc $P$ in $\check S$. We identify $\check S$ with the Lagrangian Grassmannian $\mathcal L(V)$ of $V = (\R^{n} \times \R^n, \omega)$, where $\omega(x,y)=x^\top Jy$ with
  \[
    J=\left(\begin{array}{cc}
			&I_n\\
			-I_n&\\
\end{array}\right)
 .\]
We now define an embedding $\iota: H:=SL(2,\R)^n \to G$ by
\[
  g=(g_1,\ldots,g_n)\mapsto 
\left(\begin{array}{cccccc}
a_1&&&b_1$$\\
&\ddots &&&\ddots&\\
&&a_n&&&b_n\\
c_1&&&d_1$$\\
&\ddots &&&\ddots&\\
&&c_n&&&d_n\\
\end{array}\right)
\] 
and choose $\check S_P := \iota(H).L_0$, where $L_0=\langle(e_1,e_1),(e_2,e_2),\ldots,(e_n,e_n)\rangle$ is a basepoint in $\mathcal L(V)$. (Here $e_k$ denotes the $k$th standard basis vector of $\R^n$ and we identify $\R^{2n}$ with $\R^n \times \R^n$.)  We now provide an $H$-equivariant identification $\nu: (S^1)^n \to \check S_P$:  The action of $g \in H$ on the former is given by $g. \lambda = (S^{-1}gS) \cdot \lambda$, where
\[
  S=\frac{1}{\sqrt{2}}\left(\begin{array}{cc}1&-i\\-i&1 \end{array}\right)
\]
and the $\cdot$-action is given by M\"obius transformations. In particular,
\begin{align*}
g.(1, \dots, 1) =& ((S^{-1}g_1S)\cdot 1, \dots, (S^{-1}g_nS)\cdot 1)\\
=& (S^{-1} \cdot (g_1 \cdot 1), \dots, S^{-1} \cdot (g_n \cdot 1)).
\end{align*}
On the other hand, the action on $\check S_P$ is given by
\begin{align*}
	g\cdot L_0=&\langle \big((a_1+b_1)e_1,(c_1+d_1)e_1\big),\ldots,\big((a_n+b_n)e_n,\big((c_n+d_n)e_n\big)\rangle\\
	=&\langle \big((a_1+b_1)(c_1+d_1)^{-1}e_1,e_1\big),\ldots,\big((a_n+b_n)(c_n+d_n)^{-1}e_n,\big(e_n\big)\rangle\\
	=&\langle \big((g_1\cdot 1)e_1,e_1\big),\ldots,\big((g_n\cdot 1)e_n,e_n\big)\rangle,
\end{align*}
The points $(1,\ldots,1)$ and $L_0$ have the same stabilizer in $H$. Thus the desired identification is given by
\[\nu(\lambda_1, \dots, \lambda_n) = (((S \cdot \lambda_1)e_1, e_1), \dots, ((S \cdot \lambda_n)e_n, e_n)).\]
By functoriality the pullback $\nu^*B_{\check S}$ to $(S^1)^n$ satisfies 
\begin{eqnarray*}
\left(\nu^*B_{\check S}((\lambda_1^{(1)}, \dots, \lambda_n^{(1)}), \dots, (\lambda_1^{(4)}, \dots, \lambda_n^{(4)}))\right)^n  &=& \prod_{j=1}^n \; [\lambda_j^{(1)} : \lambda_j^{(2)} : \lambda_j^{(3)} : \lambda_j^{(4)}].
\end{eqnarray*}
We now recall the definition of the classical cross ratio $B_{class}$: Given Lagrangians $L^{(j)} = \langle l_1^{(j)}, \dots, l_n^{(j)}\rangle$ for $j = 1, \dots, 4$ we define
\[
  B_{class}(L^{(1)},L^{(2)},L^{(3)},L^{(4)}):=\frac{\det(A^{12})\det(A^{34})}{\det(A^{14})\det(A^{32})},
\]
where $A^{ij}$ is the matrix given by
  \[
    A^{ij}_{ab}:=\omega(l^{(i)}_a,l^{(j)}_b).
  \]
If, in particular, $L^{(j)}$ is of the form $L^{(j)} = \langle (\alpha_1^{(j)} e_1, e_1), \dots, (\alpha_n^{(j)} e_n, e_n)\rangle$, then a direct calculation shows that
  \[
    A^{ij}_{ab}:= (\alpha_a^{(j)}-\alpha_b^{(i)})\cdot \delta_{ab}\;\Rightarrow\; \det(A^{ij}) = \prod_{k=1}^n (\alpha_k^{(j)} - \alpha_k^{(i)}),
  \]
and thus
\begin{eqnarray*}
  B_{class}(L^{(1)},L^{(2)},L^{(3)},L^{(4)}) &=&\frac{\prod_{k=1}^n (\alpha_k^{(2)} - \alpha_k^{(1)}) \prod_{k=1}^n (\alpha_k^{(4)} - \alpha_k^{(3)})}{\prod_{k=1}^n (\alpha_k^{(4)} - \alpha_k^{(1)})\prod_{k=1}^n (\alpha_k^{(2)} - \alpha_k^{(3)})}\\
  &=& \prod_{k=1}^n [\alpha_k^{(1)}:\alpha_k^{(2)}:\alpha_k^{(3)}:\alpha_k^{(4)}].
\end{eqnarray*}
In particular we finally obtain
\begin{eqnarray*}
&&\nu^*B_{class}((\lambda_1^{(1)}, \dots, \lambda_n^{(1)}), \dots, (\lambda_1^{(4)}, \dots, \lambda_n^{(4)}))\\
&=& \prod_{j=1}^n \; [S \cdot \lambda_j^{(1)} : S \cdot \lambda_j^{(2)} : S \cdot \lambda_j^{(3)} : S \cdot \lambda_j^{(4)}]\\
&=&  \prod_{j=1}^n \; [\lambda_j^{(1)} : \lambda_j^{(2)} : \lambda_j^{(3)} : \lambda_j^{(4)}],\\
\end{eqnarray*}
which establishes $(\nu^*B_{\check S})^n = \nu^*B_{class}$ and thus $B_{\check S}^n = B_{class}$.\\

We deduce that that $b_{class} = b_\rho^n$. In particular, if $n$ is even, then $b_{class} \geq 0$. It then follows that $b_{class} $ violates Axiom \eqref{a4} and thus is not a strict cross ratio in the sense of Labourie. While $b_{class} $ can obviously be recovered from our $b_\rho$, it is not completely obvious how to find a consistent $n$th root of $b_{class} $. Thus our construction contains valuable additional information even in the most classical case.
\end{example}

\subsection{Translation lengths}

We now apply cross ratios of maximal representations for estimates of the corresponding translation lengths. For this we fix a maximal representation $\rho: \Gamma \to G$ and denote by $\phi$ the associated continuous monotone limit curve. Since every $\gamma \in \Gamma\setminus\{e\}$ is hyperbolic when considered as an element of $PU(1,1)$, it has a unique attractive fixed point $\gamma^+$ and a unique repellent fixed point $\gamma^-$. We may thus define
\begin{eqnarray}
g^{\pm} := \phi(\gamma^{\pm}).
\end{eqnarray}
Then we have:
\begin{proposition}
The pair $(g^+, g^-)$ is an attractor-repellor pair for $\rho(\gamma)$.
\end{proposition}
\begin{proof} By Lemma \ref{contraction} the element $\rho(\gamma)$ contracts a dense open subset of $\check S$ to $g^+$. This implies that the corresponding element $g_1$ contracts a dense open subset of $\check S$ to $e$, and thus for every $v \in V$ we have $g_2^n.v \to \infty$. This implies that every eigenvalue of $g_2$ has modulus $>1$.
\end{proof}
In particular, we can define the associated period \[\tau_{\mathcal D}^\infty(\rho(\gamma)) := \tau_{\mathcal D}^\infty(\rho(\gamma), g^+, g^-);\] we then have for any $\xi \in S^1 \setminus\{\gamma^\pm\}$,
\begin{eqnarray}\label{PeriodCR}
\tau_{\mathcal D}^\infty(\rho(\gamma)) &=& b_\rho(\gamma^-, \xi, \gamma^+, \gamma.\xi).
\end{eqnarray}
Now we have the following special case of Corollary \ref{CRTransl}:
\begin{theorem}\label{MainEstimateCR} Let $\Sigma$ be a closed oriented surface of negative Euler characteristic and $\Gamma = \pi_1(\Sigma)$ its fundamental group. Let $G$ be a semisimple Hermitian Lie group with finite center and associated bounded symmetric domain $\mathcal D$ and $\rho: \Gamma \to G$ a maximal representation. Then there exist positive constants $C_1(\mathcal D), C_2(\mathcal D)$ depending only on $\mathcal D$ such that for all $\gamma \in \Gamma$,
\[C_1(\mathcal D) \cdot \tau_{\mathcal D}^\infty(\rho(\gamma)) \leq  
  \tau_{\mathcal D}(\rho(\gamma)) \leq  C_2(\mathcal D)  \cdot \tau_{\mathcal D}^\infty(\varrho(\gamma)),\]
where the period is given by \eqref{PeriodCR} and the translation length is taken with respect to the (unnormalized) Bergman metric on $\mathcal D$.
\end{theorem}
Indeed, if  $\mathcal D_1, \dots, \mathcal D_l$ are the irreducible factors of $\mathcal D$ then we can choose
\[C_1(\mathcal D) :=   \sqrt{\min_j \dim_\C \mathcal D_j}\]
and
\[C_2(\mathcal D) :=  2\cdot {\rm rk}\, \mathcal D \cdot \max_j \frac{\dim_\C \mathcal D_j}{{\rm rk}\, \mathcal D_j}.\]
\subsection{Well-displacing}

To establish the desired well-displacing property for maximal representations we need a version of the Milnor-\v{S}varc lemma. Given a group $\Gamma$ with finite generating set $S$ we denote by $\|\,\cdot \,\|_S$ the word length with respect to $S$ and by
\[d_S(\gamma_1, \gamma_2) := \|\gamma_2^{-1}\gamma_1\|_S\]
the associated word metric. Then the classical version of the Milnor-\v{S}varc lemma is given as follows (see {\cite[Prop. I.8.19]{BrHa}}):
\begin{lemma}[Milnor-\v{S}varc]\label{MilSva}
Let $(X,d)$ be a length space. If a group $\Gamma$ acts properly
and cocompactly by isometries on $X$, then $\Gamma$ is finitely
generated and for every finite generating set $S$ with associated
word metric $d_S$ on $\Gamma$ and every basepoint $x_0 \in X$ the
map
\[(\Gamma, d_S) \to (X,d), \quad \gamma \mapsto \gamma.x_0\]
is a quasi-isometry.
\end{lemma}
Here we are interested in the case where $\Gamma$ is the fundamental group of a closed oriented surface and $X = \mathbb D$. In this setup we will need a version of the Milnor-\v{S}varc lemma which compares the translation length $l_S$ of $(\Gamma, S)$ (as defined in \eqref{LengthFunction} on page \pageref{LengthFunction}) to the translation length of $\mathbb D$. The following inequality is sufficient for our purposes:
\begin{corollary}\label{DiscTranslationLength}
Let $S$ be an arbitrary finite generating set for $\Gamma$. Then there exist constants $A, B > 0$ such that for every $\gamma \in \Gamma$
\[\tau_{\mathbb D}(\gamma) \geq A \cdot l_S(\gamma) - B.\]
\end{corollary}
\begin{proof}
 We fix a compact fundamental domain $F$ for the $\Gamma$-action on
  $\mathbb{D}$. We know that every $\gamma\in \Gamma$ is hyperbolic, i.e. there
  exists a geodesic $\sigma$ on which $\gamma$ acts by translation
  and we have $\gamma\cdot
  \sigma(t)=\sigma(t+\tau_\mathbb{D}(\gamma))$ for all $t$. There
  exists $\eta\in \Gamma$ such that $\eta \sigma$ intersects $F$,
  say $y:=\eta \sigma(t_0)\in F$. Then we have for any $x\in F$:
  \[
    d(x,\eta\gamma\eta^{-1}x)\leq
    d(x,y)+d(y,\eta\gamma\eta^{-1}y)+d(\eta\gamma\eta^{-1}y,\eta\gamma\eta^{-1}x)\leq
    2\text{diam}(F)+\tau_\mathbb{D}(\eta\gamma\eta^{-1}).
  \]
  Now we fix $x\in F$ and apply the Milnor-\v{S}varc lemma with $x _0= x$ to find positive constants $A, B'$ satisfying
  \[d(x, \gamma x) = d(e x, \gamma x) \geq A \cdot d_S(e, \gamma) - B' = A \cdot l_S(\gamma) - B'\]
  for all $\gamma \in \Gamma$. We deduce that
  \begin{eqnarray*}
    \tau_\mathbb{D}(\gamma)&=&\tau_\mathbb{D}(\eta\gamma\eta^{-1})\geq
    d(x,\eta\gamma\eta^{-1}x)-2\text{diam}(F)\\
    &\geq &A\cdot l_S(\eta\gamma\eta^{-1})-B'-2\text{diam}(F)=A \cdot
    l_S(\gamma)-(B'+2\text{diam}(F)).
  \end{eqnarray*}
\end{proof}
Combining this with Labourie's equivalence theorem for cross ratios (in the form of Corollary \ref{CREq}) and Theorem \ref{MainEstimateCR} we then obtain well-displacing of maximal representations:
\begin{theorem}[Well-displacing]\label{WD}
Let $\Gamma$ be the fundamental group of a closed oriented surface $\Sigma$, $\mathcal D$ a bounded symmetric domain and $S$ a finite generating set for $\Gamma$. Then for every
 maximal representation $\rho: \Gamma \to G(\mathcal D)^0$ there exist $A,B>0$ such that for all $\gamma \in \Gamma$, 
\[
    \tau_{\mathcal D}(\rho(\gamma))\geq A \cdot l_S(\gamma)-B.
\] \end{theorem}
\begin{proof} Using Corollary \ref{CREq}, Theorem \ref{MainEstimateCR}, Equation (\ref{DiscEquality}) and Corollary \ref{DiscTranslationLength}
we find positive constants $C_1, \dots, C_4$ such that
\begin{eqnarray*}
\tau_{\mathcal D}(\rho(\gamma)) &\geq& C_1 \cdot \tau^\infty_{\mathcal D}(\rho(\gamma)) \\
&=& C_1 \cdot \log b_\rho(\gamma^-, \xi, \gamma^+, \gamma \xi)\\
&\geq& C_2\cdot \log [\gamma^-: \xi: \gamma^+: \gamma \xi]-1\\
&=& C_2 \cdot \tau_{\mathbb D}^\infty(\gamma)-1\\
&=& C_2 \cdot \tau_{\mathbb D}(\gamma) -1\\
&\geq& C_2C_3 \ell_S(\gamma)-C_2C_4-1.
\end{eqnarray*}
\end{proof}
Note that compactness of $\Sigma$ was indispensable for the proof of Theorem \ref{WD}.

\subsection{Proofs of Corollaries  \ref{CorQI}-\ref{CorProperness2}}\label{SubsecCorollaries}  
All three corollaries are well-known consequences of the well-displacing property established in Theorem \ref{WD}. For the convenience of the reader we provide some explicit references:\\ 

Corollary \ref{CorQI} follows from \cite[Prop. 4.2.1]{DGLM} and \cite[Lemma 4.0.4]{DGLM}, since higher genus surface groups are hyperbolic.\\

Corollary \ref{CorQI2} follows from \cite[Lemma 2.7]{AnnaMCG} (or Corollary \ref{CorQI} and the Milnor-\v{S}varc lemma) and the proof of Theorem \ref{WD}.\\

Corollary \ref{CorProperness} follows from Corollary \ref{CorQI2} and \cite[Prop. 2.4]{AnnaMCG}. \\

Finally, Corollary \ref{CorProperness2} follows from \cite[Thm. 5.2.2]{Lab05}.

\newpage

\appendix

\section{Complements to the theory of Euclidean Jordan algebras}\label{AppJordan}

Throughout this article we have made essential use of results from the theory of Euclidean Jordan algebras. Most of these results are standard and can be found in the literature, see in particular \cite{FK, BraunKoecher}. However, there are a couple of facts for which we were unable to find explicit references; for the convenience of the reader we collect these results in the present appendix.\\

Let us start by recalling two folklore results. One of the most important notions in the theory of Euclidean Jordan algebras is that of a Jordan frame \cite[p. 44]{FK}. In this context we will need the following standard lemma in Subsection \ref{SubsecBalanced}:
\begin{lemma}\label{JordanCompletion}
Let $d_1, \dots, d_m$ be a collection of pairwise orthogonal idempotents in a Euclidean Jordan algebra $V$ with $d_1 + \dots + d_m = e$. Then there exists a Jordan frame $c_1,
\dots, c_r$ of $V$ and numbers $i_1 < \dots < i_m < i_{m+1} = r$ such that
\[d_j = \sum_{l = i_j+1}^{i_{j+1}} c_l.\]
\end{lemma}
\begin{proof} Let us call a collection $(d_1, \dots, d_m)$ as in the lemma a pre-Jordan frame. By finite-dimensionality of $V$ it suffices to show the following: If $(d_1, \dots, d_m)$ is a pre-Jordan frame and $d_1 = f_1 + f_2$ with $f_1, f_2$ idempotents, then the $(m+1)$-tuple $(f_1, f_2, d_2 \dots, d_m)$ is again a pre-Jordan frame. Indeed,   $f_1 + f_2 + d_2 + \dots + d_m = e$. Moreover we have
\[f_1 + f_2 = d_1 = d_1^2 = (f_1 + f_2)^2 = f_1^2 + 2f_1f_2 + f_2^2 = f_1 + f_2 + 2f_1f_2,\]
whence $f_1f_2 = 0$. This implies in particular that
\[d_1f_j = (f_1 + f_2)f_j = f_j^2 = f_j,\]
whence $f_1, f_2$ are in the $1$-eigenspace of $d_1$, while $d_2, \dots, d_m$ are in the $0$-eigenspace of $d_1$. Then the lemma follows from the orthogonality of these eigenspaces \cite[Satz I.12.3 a)]{BraunKoecher}.
\end{proof}
Given a Euclidean Jordan algebra with associated bounded symmetric domain $\mathcal D$ one can also characterize the Shilov boundary $\check S$ of $\mathcal D$ in terms of Jordan frames. This is the content of the following proposition, which appears in the proof of  \cite[Proposition X.2.3]{FK} and will be used in the proof of Proposition \ref{Quadruples}:
\begin{proposition}\label{shilovpoints}
For every $z\in \check S$ there exists a Jordan frame $(c_1,\ldots,c_r)$ and complex numbers $\lambda_i$ with
 $|\lambda_i|=1$ such that
 \[
  z=\sum_{i=1}^r \lambda_i c_i.
 \]
\end{proposition}
While the above two results are well-known, the following more specific results seem to be new. Our first result concerning morphisms of Euclidean Jordan algebras is needed to complete the proof of Proposition \ref{tightness}.
\begin{proposition}\label{tightnessAppendix}
Let $\mathcal D_1, \mathcal D_2$ be bounded symmetric domains of tube type with respective Shilov boundaries $\check S_1$ and $\check S_2$, and $\beta: \mathcal D_1\to  \mathcal D_2$ be a boundary morphism. Then there exist Euclidean Jordan algebras $V_1$, $V_2$, a Jordan algebra homomorphism $\alpha: V_1 \to V_2$ and isomorphisms $\mathcal D_j \cong \mathcal D_{V_j}$ intertwining $\beta$ and $\alpha^\C$.
\end{proposition}
The proof uses the theory of positive Hermitian Jordan triple systems (pHJts'). We refer the reader to \cite{Clerc} for background. We recall that the unit balls of such triples systems (always with respect to the spectral norm) are circled (i.e. invariant under the diagonal multiplication with elements of $S^1$) and symmetric (see \cite[Thm. 4.1]{Loos}), and that every bounded symmetric domain arises as the unit ball of a pHJts (see  \cite[Thm. 1.6 and Thm. 4.1]{Loos} and \cite{Clerc}). Every morphism of pHJts' induces a morphism of the corresponding unit balls. Conversely we have:.
\begin{lemma}\label{LoosLemma} Let $W_1, W_2$ be positive Hermitian Jordan triple systems and $\mathcal D_1, \mathcal D_2$ their unit balls with respect to the respective spectral norms. Then every morphism $\beta: \mathcal D_1 \to \mathcal D_2$ with $\beta(0) = 0$ extends to a morphism $W_1 \to W_2$ of pHJts.
\end{lemma}
\begin{proof} We adapt an argument of Loos \cite{Loos} going back to Cartan \cite[p. 30]{Cartan} (see also \cite[L. X.5.2]{FK}): Consider the maps $\beta^{(1)}_t(z) := \beta(e^{it}z)$ and $\beta^{(2)}_t(z) := e^{it}\beta(z)$ for $t \in \R$. Since $\mathcal D_1$ and $\mathcal D_2$ are circled, these map $\mathcal D_1$ into $\mathcal D_2$; moreover, both maps are affine, since $\beta$ is, and share the same $z$-derivative at the origin. Since also $\beta^{(1)}_t(0) = \beta^{(2)}_t(0) = 0$ we deduce \cite[Prop. 3.2]{Postnikov} that $\beta^{(1)}_t = \beta^{(2)}_t$; comparing Taylor expansions, we see that $\beta$ is linear and thus extends to $\beta: W_1 \to W_2$. Since the derivative of a morphism of bounded symmetric domains is a morphism of Jordan triple systems \cite[Thm. III.2.8]{Bertram} and the exponential map intertwines the Jordan triple structures on $W_j$ and $T_0W_j$, the lemma follows.
\end{proof}
Now we can deduce Proposition \ref{tightnessAppendix}:
\begin{proof}[Proof of Proposition \ref{tightnessAppendix}]
By applying suitable isomorphisms we may assume that $\mathcal D_1$ and $\mathcal D_2$ are the unit balls of pHJts' $W_1, W_2$ with respect to the corresponding spectral norms and that $\beta(0) = 0$. Then Lemma \ref{LoosLemma} applies and provides a linear extension $\beta: W_1 \to W_2$, which is a morphism of Euclidean Jordan triple systems. Note that by uniqueness, $\beta|_{\check S_1}$ is the boundary extension of $\beta$. Since $\mathcal D_1$ and $\mathcal D_2$ are of tube type, the elements of $\check S_j$ are precisely the maximal tripotents of the Jordan triple system $W_j$ \cite[Thm. 4.2]{Clerc}. Now pick $e_1 \in \check S_1$ arbitrarily and define $e_2 := \beta(e_1)$. Out of the respective triple products $\{\cdot, \cdot, \cdot\}$ we then obtain complex Jordan algebra structures on $W_1$ and $W_2$ by \[x \cdot y := \{x,e_j,y\};\] by construction, $\beta$ is a morphism $(W_1, \cdot) \to (W_2, \cdot)$ and maps the Euclidean real forms given by
\[V_{j}:= \{z \in W_j\,|\, \{e_j,z,e_j\} = z\}\]
to each other. Then the restriction $\alpha: V_1 \to V_2$ is the desired morphism of Euclidean Jordan algebras with $\alpha^\C|_{\mathcal D_1} = \beta$.
\end{proof}
Our final goal is to express the notion of transversality for Shilov boundaries of bounded symmetric domains of tube type in Jordan theoretic terms. For this we denote by
\[K: V^\C \times V^\C \to {\rm End}(V^\C)\]
the automorphy kernel of $V$. Given $x\in V^\C$ let $L_0(x)$ be the restriction of $L(x)$ to the subalgebra generated by all powers of $x$ and denote by $\det_V(x):=\det(L_0(x))$ the \emph{Jordan algebra determinant} of $x$ (see \cite[Ch. II.2]{FK}). Then we have the following characterization of transversality:
\begin{proposition}\label{TransMain}
Let $V$ be a Euclidean Jordan algebra, $\mathcal D$ the associated
bounded symmetric domain and $\check S$ its Shilov boundary. Then
$z,w \in \check S$ are transverse iff one of the following equivalent
conditions holds true:
\begin{itemize}
\item[(i)] $\det_V(z-w) \neq 0$.
\item[(ii)] $K(z,w)$ is invertible.
\item[(iii)] $K(z,w) \in {\rm Str}(V^\C)$.
\item[(iv)] $\det K(z,w) \neq 0$.
\end{itemize}
\end{proposition}
The lion's share of the proof is provided in \cite{CO2}. In order to complete the arguments given there, we need to understand the transformation behavior of the automorphy kernel. For this we remark that by  \cite[Ch. II, Sec. 5]{Satake} there exists a function $J:G \times \overline{\mathcal D} \to {\rm
Str}(V^\C)$, called the \emph{canonical automorphy factor}, satisfying
\begin{eqnarray}\label{AutomorphyI}K(gz, gw) = J(g,z)K(z,w)J(g,w)^*\end{eqnarray}
for $g \in G$, $z,w \in \overline{\mathcal D}$. 
\begin{proof}[Proof of Proposition \ref{TransMain}]
Let us first prove equivalence of the statements (i)-(iv): The implication (i) $\Rightarrow$ (ii) is provided in \cite[Lemma 5.1]{CO2}. The implication (ii) $\Rightarrow$ (iii) follows from the fact that $K(z_0,w_0) \in {\rm
Str}(V^\C)$ for $z_0, w_0 \in \mathcal D$ together with the continuity of $K$ and the fact that ${\rm Str}(V^\C)$ is closed in ${\rm GL}(V^\C)$. Finally, the implication (iii)
$\Rightarrow$ (iv) is obvious. Thus it remains to show (iv) $\Rightarrow$ (i). Thus let $w,z \in \check S$ be arbitrary and
assume $\det K(z,w) \neq 0$. We first claim that there exists $g \in G$ such that $e-g\cdot w$ and $e-g\cdot z$ are
invertible. Indeed, if $\mathcal D$ is a polydisc with $w=(w_1,\ldots,w_r)$ and $z=(z_1,\ldots,z_r)$, then one can clearly
find an element $g=(g_1,\ldots,g_r)\in SO(2)^r\subset G$ such that $g_i\cdot w_i$ and $g_i\cdot z_i$ are both not equal to
$1$. Since any triple of points in the Shilov boundary is contained in the boundary of a common polydisc \cite[Thm. 3.1]{ClercNeeb}, the general case can be reduced to this, thereby
finishing the proof of the claim. Next observe that \eqref{AutomorphyI} implies
    \[
     \det K(gz,gw)=\underbrace{\det(J(g,z))}_{\neq 0}\det K(z,w)\underbrace{\overline{\det(J(g,w))}}_{\neq 0},
    \]
and thus our assumption yields $\det K(gz, gw) \neq 0$. Now note that with $e-gw$ also $e-\overline{gw} = \overline{e-gw}$ is invertible , hence
\cite[Lemma X.4.4 ii)]{FK} applies and yields
    \[
        K(gz,gw)=P(e-gz)P(c(gz)+c(\overline{gw}))P(e-\overline{gw}),
    \]
whence
\[\det(P(c(gz)+c(\overline{gw})))\neq 0.\]
A simple calculation shows that $c(\overline{gw}))=-\overline{c(gw)}$. Since
$gw \in \check S\cap D(c)$, the image $c(gw)$ is contained in $V$, whence $\overline{c(gw)}=c(gw)$. We thus obtain
\[\det(P(c(gz)-c(gw)))\neq 0.\]
Using the definition of the Cayley transform and \cite[p.190]{FK} we obtain
        \begin{align*}
          c(gz)-c(gw)=&
          i\big((e+gz)(e-zg)^{-1}-(e+gw)(e-gw)^{-1}\big)\\
          =& i\big(-ie+2i(e-gz)^{-1}+ie-2i(e-gw)^{-1}\big)\\
          =&-2\big((e-gz)^{-1}-(e-gw)^{-1}\big).
        \end{align*}
We thus obtain
        \begin{eqnarray}\label{Transv2}
            \det P(-2((e-gz)^{-1}-(e-gw)^{-1}))\neq 0.
        \end{eqnarray}
    Now we can apply Hua's formula \cite[Lemma X.4.4]{FK} to obtain
        \begin{align*}
            &P(-2((e-gz)^{-1}-(e-gw)^{-1}))\\=&P(e-gz)^{-1}P(-2((e-gz)-(e-gw)))P(e-gw)^{-1}\\
            =&P(e-gz)^{-1}P(-2(gw-gz)))P(e-gw)^{-1}.
        \end{align*}
    Combinining this with \eqref{Transv2} and using that $P(e-gz)^{-1}$ and $P(e-gw)^{-1}$ are invertible, we obtain
    \[\det P(-2(gw-gz))) \neq 0.\]
    Thus $P(-2(gw-gz))$ is invertible. By \cite[Prop. II.3.1]{FK} this implies that $-2(gw-gz)$ and hence $gz-gw$ is
    invertible. Thus $\det_V(gz-gw) \neq 0$, which by \cite[Prop.
    3.2]{CO2} implies $\det_V(z-w) \neq 0$. This finishes the proof of the equivalence of (i)-(iv).\\
    
    We deduce in particular that
    \[\check S^{[2]} := \{(z,w) \in \check S\,|\, \det K(z,w) \neq 0\} = \{(z,w) \in \check S\,|\, \det{}_V(z-w) \neq 0\}\]
    The first description together with  \eqref{AutomorphyI} and the continuity of $\det K(\cdot, \cdot)$
    imply already that $\check S^{[2]}$ is $G$-invariant and open; the second description together with \cite[Prop. 3.4]{CO2} shows that $\check S^{(2)}$ is even a $G$-orbit. Since $\check S^{[2]}$ is the unique open $G$-orbit in $\check S^2$ we obtain $\check S^{(2)} = \check S^{[2]}$, which finishes the proof.
\end{proof}
Proposition \ref{TransMain} implies immediately:
\begin{corollary} The image $p(V)$ of $V$ under the inverse Cayley transform is precisely the subset of points in $\check S$, which are transverse to $e$.
\end{corollary}

\section{A version of Labourie's equivalence theorem}\label{AppLabourie}
We recall that a continuous $\Gamma$-invariant functions on $(S^1)^{4*}$ satisfying the identities
\eqref{a1}-\eqref{a5} above is called a strict cross ratio. It was observed by Labourie in \cite{Lab05} that all such strict cross ratios are essentially equivalent. This notion can be made precise in various ways; we will need the following version:
\begin{theorem}[Labourie]\label{PropLab}
  Let $b_1$ and $b_2$ be strict cross ratios. Then there exist $C,D>0$ such that
  \[D^{-1}|\log b_1|-1\leq |\log b_2|\leq C|\log b_1|+C\]
\end{theorem}
Since this formulation is slighlty different from the one provided in \cite{Lab05}, we include a complete proof. All the essential ideas are taken from \cite{Lab05}.\\

Let $b: (S^1)^{4*} \to \R$ be any strict cross ratio. We will ocassionally use the following two cocycle identities:
\begin{eqnarray}
  \label{p1} b(x,y,x,t)&=&b(x,y,z,t)b(z,y,x,t)=1\\ 
  \label{p2} \log b(x,y,z,t)&=&-\log b(z,y,x,t)
\end{eqnarray}
The former is an immediate consequence of \eqref{a4} and
\eqref{a3} and the latter follows by applying the logarithm. We
will usually consider $x,y,z$ fixed and study 
\[g(t):=b(x,y,z,t)\] as
a function of $t$. Let us assume that $(x,y,z)$ is positively
oriented. We then divide the circle into three open disjoint
intervals $I_1=(x,y)$, $I_2=(y,z)$ and $I_3=(z,x)$ so that
\[S^1 = \{x\} \cup I_1 \cup \{y\} \cup I_2 \cup\{z\} \cup I_3.\]
The function $g$ is then defined on $S^1 \setminus \{z\}$.  By Axiom \eqref{a5}, $x$ is the only zero of $g$. Since $g(y)=1$ is positive, $g$ is positive on $I_u:=I_1\cup \{y\} \cup I_2$. Let $a,b\in I_u$ such that $g(a)=g(b)$. Then we have:
\[
 1=g(a)g(b)^{-1}=b(x,y,z,a)b(x,y,z,b)^{-1}=b(-1,a,1,b),
\]
but by Axiom \eqref{a4} this can only be the case if $a=b$. Hence $g|_{I_u}$ is injective. Furthermore, by Axiom \eqref{a5} and \eqref{p1} we have
\[
 \underset{t \in I_u}{\lim_{t\rightarrow z}} g(t)=\underset{t \in I_u}{\lim_{t\rightarrow z}}b(z,y,x,t)^{-1}= +\infty.
\]
Since $g(x) = 0$ the intermediate value theorem implies that $g|_{I_u}$ is surjective and therefore defines a homeomorphism between $I_u$ and $(0,\infty)$. Now consider the case
$t\in I_3$. We claim that $g(t)$ is
negative on $I_3$. For this we first observe that as above
\[\underset{t \in I_u}{\lim_{t\rightarrow z}} g(t) \in \{\pm \infty\}.\]
Again, since $g(x) = 0$ the intermediate value theorem implies that $g$ maps $I_3$ homeomorphically to either $(0, \infty)$ or $(-\infty,0)$. However, the former would imply the
existence of $t_0 \in I_3$ with $g(t_0) = 1$, which contradicts Axiom \eqref{a4}. Hence
\[\underset{t \in I_u}{\lim_{t\rightarrow z}} g(t) = -\infty\]
and $g$ maps $I_3$ homeomorphically to $(-\infty,0)$. We have proved:
\begin{proposition}\label{StrictCrHomeo}
If $(x,y,z)$ is a positively oriented triple on $S^1$ and $b: (S^1)^{4*} \to \R$ is a strict cross ratio, then
 \[g: S^1 \setminus\{z\} \to \R, \quad t\mapsto b(x,y,z,t)\]
is a homeomorphism with $g(x) = 0$.
\end{proposition}
Notice that as a homeomorphism $S^1 \setminus \{z\} \to \R$ the
function $g$ is automatically monotonous. In the sequel we denote
by
\[(S^1)^{3+} := \{(x,y,z) \in (S^1)^{(3)}\,|\, (x,y,z) \text{ positively ordered}\}\]
the set of positively ordered triples.

\begin{corollary}\label{FlowOfACr}
  For every strict cross ratio there exists a $\Gamma$-equivariant continuous map (with the trivial action on $\R$)
  \[
    \psi:\R\times (S^1)^{3+}\rightarrow S^1
  \]
  such that $\log b(x,y,z,\psi_s(x,y,z))=s$. This function satisfies
  \begin{eqnarray}\label{CompatibleFlow}
   \psi_{s+t}(x,y,z) = \psi_t(x,\psi_s(x,y,z), z).
  \end{eqnarray}
\end{corollary}
\begin{proof}
  Put $\psi_s(x,y,z):=(\log g)^{-1}(s)$, where $g$ is as in the last proposition. Then $\log b(x,y,z,\psi_s(x,y,z))=s$ holds by definition. Moreover, abbreviating $a := \psi_s(x,y,z)$ and $b := \psi_t(x,a,z)$ we get
  \begin{eqnarray*}
  \log g(b) = \log b(x,y,z,b) &=& \log b(x,y,z,a) + \log b(x,a,z,b)\\
  &=&  \log b(x,y,z,\psi_s(x,y,z)) +  \log b(x,a,z,\psi_t(x,a,z))\\
  &=& s +t,
  \end{eqnarray*}
  whence $b = (\log g)^{-1}(s+t) = \psi_{s+t}(x,y,z)$ as claimed.
\end{proof}
Now we can deduce the theorem:
\begin{proof}[Proof of Theorem \ref{PropLab}] Let $\psi^1_s$ and $\psi^2_s$ be maps associated to $b_1$ and $b_2$ by means
of Corollary \ref{FlowOfACr}. Define a function $T:(S^1)^{3+} \rightarrow \R$ by
\[
 T(x,y,z):= \log b_1(x,y,z,\psi^2_1(x,y,z))
\]
    This map is positive and continuous. Since $ \psi^2_s$ is $\Gamma$-equivariant,
    T is $\Gamma$-invariant. Furthermore it satisfies 
    \begin{eqnarray}\label{FlowIndStart} \psi_1^2(x,y,z)=\psi^1_{T(x,y,z)}(x,y,z). \end{eqnarray}
    Since $(S^1)^{3+}/\Gamma$ is compact, $|T|$ has a global maximum $A$. Now consider the function $f: \R \to \R$ (depending on $x,y,z$) given by
\[
  f(s) := \log b_1(x,y,z, \psi^2_{s}(x,y,z))
\]    
For $n\in \Z$ we have 
\begin{align*}
  |f(n)|=&\left| \log b_1(x,y,z,\psi_n^2(x,y,z))\right|\\
  =& \left|\sum_{i=0}^{n-1}\log b_1(x,\psi_i(x,y,z),z,\psi_{i+1}(x,y,z))\right|\\
  =&\left|\sum_{i=0}^{n-1}\log b_1(x,\psi_i(x,y,z),z,\psi_1(x,\psi_i(x,y,z),z))\right|\\
  \leq &A\cdot |n|
\end{align*}
Because of monotonicity of $f$ we get for $0\leq s\in [n,n+1)$:
\[
  f(s)\leq f(n+1)\leq An+A\leq As+A
\]
and for $0\geq s\in [n,n+1)$:
\[
  |f(s)|\leq |f(n)|\leq A|n|\leq A(|s|+1)=A|s|+A. 
\]
We can summarize these inequalities to $|f(s)|\leq A\cdot |s|+A$ and we get
\begin{eqnarray*} 
 |\log b_1(x,y,z,t)|
      &=&|\log b_1(x,y,z, \psi^2_{s}(x,y,z))| = |f(s)|\\
     &\leq&A \cdot |s|+A\\
     &=& A \cdot |\log b_2(x,y,z,t)|+A.
\end{eqnarray*}    
This proves the upper bound for $(x,y,z)\in (S^1)^{3+}$. If
$(x,y,z)$ is negatively oriented, then $(z,y,x)$ is in
$(S^1)^{3+}$. The upper bound for this case follows from the fact
that:
\[
 |\log b(x,y,z,t)|=|\log b(z,y,x,t)|.
\]
The lower bound is obtained by reversing the roles of $b_1$ and $b_2$.
\end{proof}
\begin{remark}
  Note that the compactness of $\Sigma$ is crucial for the proof of Theorem \ref{PropLab}.
\end{remark}
\section{Uniqueness of limit curves}\label{AppLimitCurve}
The purpose of this appendix is to provide a detailed proof of Proposition \ref{LCUnique}, which claims that the continuous monotone limit curve associated with a maximal representation is unique. Throughout this appendix we fix a maximal representation $\varrho: \Gamma \to G$. Our starting point is the following observation:
\begin{lemma}\label{LCintersect}
Let $\varphi_1, \varphi_2: S^1\rightarrow \check S$ be two continuous monotone limit curves for the same maximal representation $\varrho$. If $\varphi_1(S^1) \cap \varphi_2(S^1) \neq \emptyset$, then
$\varphi_1 = \varphi_2$.
\end{lemma}
\begin{proof}  By equivariance of the
$\varphi_j$ the intersection contains a $\Gamma$-orbits, but since the $\Gamma$-action on $S^1$ is minimal this implies that this
preimage is the full circle and thus $\varphi_1(S^1) =
\varphi_2(S^1)$. Every $\gamma \in \Gamma$ has a unique attractive fixed point $\gamma^+$ in $S^1$. By equivariance, this is
mapped under both $\varphi_j$ to the unique attractive fixed point of $\rho(\gamma)$ in $\varphi_1(S^1) =
\varphi_2(S^1)$. We deduce $\varphi_1(\gamma^+) = \varphi_2(\gamma^+)$ for all $\gamma \in \Gamma$ and since $\{\gamma^+\,|\,
\gamma \in \Gamma\}$ is dense in $S^1$ we have $\varphi_1 = \varphi_2$.
\end{proof}
It thus remains to show that any two continuous monotone limit curves intersect. As pointed out to us by Olivier Guichard, this fact can be derived from a general contraction property of Anosov representations. To formulate the contraction property, let $\gamma \in \Gamma-\{\id\}$ and denote by $\gamma^-$ the unique repellent and by $\gamma^+$ the unique attractive fixed point of $\gamma$ in $S^1$. Then we have:
\begin{lemma}\label{contraction}
  Let $\gamma\in \Gamma-\{\id\}$ and $\gamma^+\in S^1$ its attractive fixed point. Then for any limit curve $\varphi$ the sequence $\varrho(\gamma)^n$ contracts an open and dense set $U = U(\phi, \gamma)$ of the Shilov boundary to $\varphi(\gamma^+)$.
\end{lemma}
Let us first ensure that this indeed yields the desired conclusion:
\begin{proof}[Proof of Proposition \ref{LCUnique}]
Assume $\phi_1$ and $\phi_2$ are two limit curves for the maximal representation $\varrho$ and let $x \in U = U(\phi_1, \gamma) \cap U(\phi_2, \gamma)$, which is non-empty by the lemma. Then $\varrho(\gamma)^n x$ converges to both $\varphi_1(\gamma^+)$ and $\varphi_2(\gamma^+)$, whence $\varphi_1(\gamma^+) = \varphi_2(\gamma^+)$. This shows that the two limit curves intersect, whence coincide by Lemma \ref{LCintersect}. 
\end{proof}
It thus remains to deduce Lemma \ref{contraction} from the Anosov property of $\varrho$. 
Throughout our discussion we fix a maximal representation $\rho$, a continuous limit curve $\phi$ and an element $\gamma \in \Gamma \setminus\{e\}$. We denote by $\gamma^+$ and $\gamma^-$ its unique attractive respectively repellent fixed point in $S^1$.  We also assume that the bounded symmetric domain $\mathcal D$ associated with $G$ has been realized as $\mathcal D = \mathcal D_V$ for some formally real Jordan algebra $V$. Moreover, we will assume $\varphi(\gamma^{\pm})  =\pm e_V$, so that 
\begin{eqnarray}\label{GammaLevi} \varrho(\gamma) \in Q_+\cap Q_-,
\end{eqnarray}
where $Q_{\pm} = Q_{\pm, V}$ is our standard pair of Shilov parabolics. Since the situation of Lemma \ref{contraction} is conjugation-invariant, this is not a restriction, and it will simplify our notation. We abbreviate by $M := T^1\Sigma$ the unit tangent bundle of $\Sigma$ and by $\overline{M} := T^1{\widetilde \Sigma}$ the unit tangent bundle of its universal covering. Then $M = \Gamma\backslash\overline{M}$ and we denote by $p: \overline{M} \to M$ the canonical projection, which is induced by the natural $\Gamma$-action on $\overline{M}$. Note that the geodesic flows $\bar\varphi_t$ and $\varphi_t$ on $\overline{M}$ respectively $M$ are related by the formula
\[\varphi_t(\Gamma x) = \Gamma\bar\varphi_t(x)\quad (x \in \overline{M}).\]
We recall our notation $\check S^{(2)}$ for the space of transverse pairs in the Shilov boundary; we identify $\check S^{(2)}$ with the $G$-orbit of $(eQ_-, eQ_+)$ in $G/Q_- \times G/Q_+$. We now define a $\check S^{(2)}$-bundle $E_{\rho} \to M$ by
\[
  E_{\rho} := \Gamma\backslash (\overline{M} \times \check S^{(2)})\to M,
\]
where the action on the first factor is by covering transformations, while the action of the second factor is induced by $\rho$. Since
\[
\overline{E_{\rho}} := p^*E_{\rho} = \overline{M} \times \check S^{(2)}.
\]
is trivial, the bundle $E_{\rho}$ is flat. The flow $\bar \phi_t$ extends to a flow $\overline{\hat \varphi_t}$ on $\overline{E_{\rho}}$ by
\[
  \overline{\hat \varphi_t}(v,s):=(\bar\varphi_t(v),s).
\]
This flow descends to a flow $\hat \varphi_t$ on $E_{\rho}$, which lifts the geodesic flow $\varphi_t$. Now the product structure of $\overline{E_{\rho}}$ induces a splitting
\[T\overline{E_{\rho}} \cong  T{\overline M} \oplus T{\check S^{(2)}} \cong T{\overline M} \oplus (T\check S \oplus T\check S)|_{\check S^{(2)}}.\]
To distinguish the second and the third summand in the last decomposition we denote them by $\bar E_\varrho^+$ and $\bar E_\varrho^-$ respectively. By definition the fiber of $\bar E_\varrho^\pm$ over $(v, s^+, s^-)$ is $T_{s^{\pm}}\check S$ and both bundles are invariant under $\overline{\hat \varphi_t}$. This implies that the bundles
\[
E^\pm_\varrho:=\Gamma\backslash \bar E^\pm_\varrho,
\]
are invariant under the flow $\hat \varphi_t$. We use the notation $p_{\varrho}^{\pm}: E_{\varrho}^{\pm} \to E_{\varrho}$ for the projections. Now we bring into play our limit curve $\varphi$. Here we use the fact that the space $\overline{M}$ may be parametrized by positive triples in $S^1$ in such a way that $(v_-, v_0, v_+) \in (S^1)^3$ parametrizes the projection of $v_0$ onto the geodesic $v_-v_+$. 
\begin{lemma}\label{LimitCurveSection} If $\phi$ is a limit curve, then the function 
 \[\bar{\sigma}_\phi: \overline{M} \to \overline{M} \times \check S^{(2)}, \quad v = (v_-, v_0, v_+) \mapsto (v,(\phi(v_-),\phi(v_+)))\]
 is $\bar\varphi_t$-invariant and $\Gamma$-equivariant. It descends to a continuous section $\sigma_\phi: M \to E_\rho$ of the bundle $E_{\rho}$, which is $\varphi_t$-invariant. \end{lemma}
\begin{proof} The function is well-defined by monotonicity of $\phi$ and clearly a section of $\overline{E_\rho}$. It is $\Gamma$-equivariant, since $\phi$ is  $\rho$-equivariant, and flow-invariant, since $\bar{\sigma}_\phi$ does not depend on $v_0$ in the first coordinate.
\end{proof}
The sections $\bar \sigma_\phi$ and $\sigma_\phi$ allow us to define bundles $\bar \sigma_\phi^*\overline{E^\pm}_\varrho$ and $\sigma_\phi^*E^\pm_\varrho$ over $\overline{M}$ and $M$ respectively. These bundles are related by the formula
\[
 \Gamma \backslash(\bar \sigma_\varphi^*\bar E^\pm_\varrho)=\sigma_\varphi^*E^\pm_\varrho.
\]
From the explicit description
\[
\sigma_\varphi^*E^\pm_\varrho = \{(m,e) \in M \times E^\pm_\varrho\,|\, \sigma_{\varphi}(m) = p_\rho^{\pm}(e) \}
\]
we see that the bundles $\sigma_\varphi^*E^\pm$ are invariant under the flow $\psi_t := \phi_t \times \hat \phi_t$ on $M \times E^\pm_\varrho$. We denote by $\bar \psi_t$ the corresponding flow on $\bar \sigma_\varphi^*\bar E^\pm_\varrho$. These flows lift the geodesic flows $\bar \phi_t$ and $\phi_t$. We introduce the notations 
\[\pi_+: \sigma_\varphi^*E_\rho^+ \to M, \quad \pi_-:\sigma_\varphi^*E_\rho^- \to M\]
for the canonical projections. Now the main technical result of \cite{LimitCurves} reads as follows:
\begin{lemma}[Burger-Iozzi-Wienhard]\label{BIWAnosov}
The section $\sigma_\phi: M \to E_\rho$ is an Anosov section, i.e. for any continuous family of norms $(\|\cdot\|_m)_{m \in M}$ on $\sigma_\phi^* E_\rho^{\pm}$ there exist constants $A,a > 0$ such that for every $m \in M$, $v^{\pm} \in (\sigma_\phi^* E_\rho^{\pm})_m$ and $t > 0$,
\begin{eqnarray*}
\|\psi_{\pm t}(v^\pm)\|_{\pi_{\pm}(\psi_{\pm t}(v^\pm))} \leq A \exp(-at) \|v^{\pm}\|_m.
\end{eqnarray*}
\end{lemma}
In order to deduce Lemma \ref{contraction} we need to relate the contraction property of the flow  $\psi_t$ to a similar contraction property of the $\Gamma$-action. Since $\phi(\gamma^+)$ is a fixed point of $\rho(\gamma)$, the element $\rho(\gamma) \in G$ acts on $T_{\phi(\gamma^+)}\check S^{(2)}$. We want to identify this tangent space with a fiber of the bundles appearing in Lemma \ref{BIWAnosov}. For this we first observe that since $\gamma$ is hyperbolic there exists a unit speed geodesic $\omega$ in $\tilde \Sigma$, such that \[
 \gamma\cdot \omega(t)=\omega(t+\tau), \quad(t\in \R),
\]
where $\tau := \tau_\D(\gamma)$ is the translation length of $\gamma$. We then have $\omega(\pm \infty) = \gamma^{\pm}$. Now denote by $\dot \omega(t)\in \overline M$ the derivative of $\omega$. Note that in terms of the geodesic flow $\bar \phi_t$ on $\overline M$ we have
\begin{equation}
 d\gamma\cdot \dot \omega(t)=\dot \omega(t+\tau)=\phi_\tau(\dot\omega(t)).\label{EqActionFlow}
\end{equation}
We see from the definition of $\bar \sigma_\phi$ that the fibers of $\bar \sigma_\varphi^*\bar E^\pm_\varrho$ along $\dot\omega(t)$ are canonically isomorphic with $T_{\varphi(\gamma^+)}\check S$. Given $t \in \R$ we introduce the notation $\iota_t: (\bar \sigma_\varphi^*\bar E^\pm_\varrho)_{\dot\omega(t)} \to T_{\varphi(\gamma^+)}\check S$ for the canonical isomorphism. These isomorphisms intertwine the infinitesimal action of $\Gamma$ on $T_{\varphi(\gamma^+)}\check S$ with the natural action on $\bar \sigma_\varphi^*\bar E^\pm_\varrho$, i.e.
\begin{eqnarray}\label{FiberIsos}
d\rho(\gamma).\iota_t(x) = \iota_{t + \tau}(\rho(\gamma).x) \quad (x \in (\bar \sigma_\varphi^*\bar E^\pm_\varrho)_{\dot\omega(t)}, t \in \R).
\end{eqnarray}
Now we choose a continuous family of norms continuous family of norms $(\|\cdot\|_m)_{m \in M}$ on $\sigma_\phi^* E_\rho^{\pm}$. We lift these norms to  the bundles $p^*E_\varrho^\pm$ over $\overline M$ by putting $\|\cdot \|_{\bar m}:=\|\cdot \|_{p(\bar m)}$, where $p:\overline M\rightarrow M$ is the canonical projection. 
\begin{lemma}\label{LContrTangent}
For every $v \in T_{\varphi(\gamma^+)}\check S$ we have
\[\lim_{n \to \infty} (d\varrho(\gamma))^nv = 0.\]
\end{lemma}
\begin{proof} We use the abbreviation $\|\cdot\|_t := \|\cdot\|_{\dot\omega(t)}$ for $t \in \R$ and define a norm on $T_{\varphi(\gamma^+)}\check S$ by 
\[\|v\| := \|\iota_0^{-1}(v)\|_0.\]
We note that the isomorphism 
\[\iota_0^{-1}\iota_\tau: (\bar \sigma_\varphi^*\bar E^\pm_\varrho)_{\dot\omega(\tau)}  \to (\bar \sigma_\varphi^*\bar E^\pm_\varrho)_{\dot\omega(0)} \]
is induced by the action of $\gamma$, hence isometric. Now let $v \in T_{\varphi(\gamma^+)}\check S$ and $x = \iota_0^{-1}(v)$. According to Lemma \ref{BIWAnosov} we thus find positive constants $a, A$ such that
\begin{eqnarray*}
\|d\rho(\gamma).v\| &=&  \|d\rho(\gamma).\iota_0(x))\| = \|\iota_\tau(\rho(\gamma).x))\| =  \|\iota_0^{-1}\iota_\tau(\rho(\gamma).x))\|_0 \\
&=& \|\rho(\gamma).x\|_\tau = \|\bar \psi_\tau(x)\|_{\tau}=\|\psi_\tau(\Gamma x)\|_{p(\dot\omega(\tau))}  \\ &\leq& A\exp(-a\tau)\|\Gamma x\|_{p(\dot\omega(0))} 
= A\exp(-a\tau)\|x\|_{0}  \\ &=&   A\exp(-a\tau)\|v\|. 
\end{eqnarray*}
Replacing $\gamma$ by $\gamma^n$ we obtain $\|d\rho(\gamma)^n.v\| \leq A\exp(-an\tau)\|v\|$ (since the constants $A,a$ are universal). This shows that $d\rho(\gamma)$ contracts $T_{\varphi(\gamma^+)}\check S$.
\end{proof}
Now we can finally deduce:
\begin{proof}[Proof of Lemma \ref{contraction}] 
We denote by $N_{\pm}$ the unipotent radicals of $Q_{\pm}$ and by $L(Q_+)$ the common Levi factor of $Q_{\pm}$. Then $N_-L(Q_+)N_+$ is open and dense in $G$, and $N_-$ is abelian (Corollary \ref{LeviRep}), hence the exponential function $\L n_- \to N_-$ is onto. In particular, the map $\iota: \L n_- \to G/Q_+$ sending $X$ to $\exp(X)Q_+$ is a homeomorphism onto an dense open subset $U:= U(\phi, \gamma)$ of $\check S$ (namely the set of all points transverse to $-e_V$). On the infinitesimal level we get an identification $\iota_*: \L n_- \to T_{e_V}\check S$. In the description of $ T_{e_V}\check S$ by equivalence classes of smooth curves it is explicitly given by
\[\iota_*(X)= [\exp(tX).e_V].\]
We recall our assumption that $e_V = \phi(\gamma^+)$ is a fixed point of $\rho(\gamma)$. In particular, $\rho(\gamma)$ acts on $T_{e_V}\check S$ and by Lemma \ref{LContrTangent} this action is contracting. Now for every $X \in \L n_-$,
\begin{eqnarray*}
\rho(\gamma)^n.\iota_*(X) &=& [\rho(\gamma)^n.\exp(tX).e_V] = [\rho(\gamma)^n.\exp(tX).\rho(\gamma)^{-n}.e_V]\\ &=& [\exp(t{\rm Ad}(\rho(\gamma))^n(X)).e_V]= \iota_*({\rm Ad}(\rho(\gamma))^n(X)),
\end{eqnarray*}
showing that $\iota_*$ intertwines the action of $\rho(\gamma)^n$ on $T_{e_V}\check S$ with the adjoint action on $\L n_-$. In particular, $\L n_-$ is invariant under ${\rm Ad}(\rho(\gamma))$ and contracted by ${\rm Ad}(\rho(\gamma))^n$. Now every $x \in U$ can be written as $x = \exp(X).e_V$ for some $X \in \L n_-$ and we have
\[\rho(\gamma)^n.x = \exp({\rm Ad}{(\rho(\gamma))^n}(X)).e_V \to \exp(0).e_V = \phi(\gamma^+).\]
Thus $\rho(\gamma)$ contracts $U$ to $\phi(\gamma^+)$ as claimed.
\end{proof}
\section{List of notations}
\addnotation b_\rho: {cross ratio associated with the maximal representation $\rho$}{CR}
\addnotation B_\Dca: {generalized cross ratio on $\Dca$}{BDca}
\addnotation B_\Dca^{(\alpha)}: {weighted Bergman cross ratio on $\mathcal D$}{BDcaalpha}
\addnotation B_{\check S}: {generalized cross ratio of $\check S$}{BS}
\addnotation \Dca: {bounded symmetric domain}{Dca}
\addnotation \Dca_V: {bounded symmetric domain of a  EuclideanJordan algebra $V$}{DcaV}
\addnotation G: {identity component of the group of biholomorphisms of $\Dca$}{G}
\addnotation \Gamma: {fundamental group of $\Sigma$}{Gamma}
\addnotation K: {stabilizer of $0$ in $G$}{K}
\addnotation k_\Dca: {Bergman kernel of $\Dca$}{kDca}
\addnotation k_V: {normalized kernel function}{kV}
\addnotation k_\chi: {kernel function associated with a character $\chi$}{kchi}
\addnotation \Omega_V: {interior of the cone of squares in a Jordan algebra $V$}{Omega}
\addnotation Q_{\pm}: {standard pair of Shilov parabolics of $G$}{Q}
\addnotation \phi: {limit curve of maximal representation}{LC}
\addnotation \rho: {maximal representation}{MR}
\addnotation {\check S}: {Shilov boundary}{shilov}
\addnotation \Sigma: {closed oriented surface}{Sigma}
\addnotation \tau_\Dca: {translation length}{tau}
\addnotation \tau^\infty_\Dca: {period}{tauinfty}
\addnotation T_\Omega: {tube $V+i\Omega$ over $\Omega$}{TOmega}
\addnotation V: {Euclidean Jordan algebra}{V}

\bibliography{Bibliography}

\end{document}